\documentclass[11pt,twoside,a4paper]{article}
\usepackage{color,amscd,amsmath,amssymb,amsthm,latexsym,stmaryrd,authblk,mathabx,shuffle,hyperref,tikz,tkz-tab} 
\usepackage[latin1]{inputenc}
\usepackage[T1]{fontenc}   
\usepackage[english]{babel}
\providecommand{\keywords}[1]{\textbf{\textit{Keywords.}} #1}
\providecommand{\AMSclass}[1]{\textbf{\textit{AMS classification.}} #1}

\newcommand{\grun}{\begin{tikzpicture}[line cap=round,line join=round,>=triangle 45,x=0.5cm,y=0.5cm]
\clip(-0.2,-0.1) rectangle (0.2,0.2);
\begin{scriptsize}
\draw [fill=black] (0.,0.) circle (1pt);
\end{scriptsize}
\end{tikzpicture}}

\newcommand{\grdeux}{\begin{tikzpicture}[line cap=round,line join=round,>=triangle 45,x=0.5cm,y=0.5cm]
\clip(-.2,-.1) rectangle (0.2,0.7);
\draw [line width=.5pt] (0.,0.5)-- (0.,0.);
\begin{scriptsize}
\draw [fill=black] (0.,0.) circle (1pt);
\draw [fill=black] (0.,0.5) circle (1pt);
\end{scriptsize}
\end{tikzpicture}}

\newcommand{\grtroisun}{\begin{tikzpicture}[line cap=round,line join=round,>=triangle 45,x=0.5cm,y=0.5cm]
\clip(-0.5,-0.1) rectangle (0.5,0.7);
\draw [line width=0.5pt] (0.,0.)-- (-0.3,0.5);
\draw [line width=0.5pt] (0.,0.)-- (0.3,0.5);
\draw [line width=0.5pt] (-0.3,0.5)-- (0.3,0.5);
\begin{scriptsize}
\draw [fill=black] (-0.3,0.5) circle (1pt);
\draw [fill=black] (0.,0.) circle (1pt);
\draw [fill=black] (0.3,0.5) circle (1pt);
\end{scriptsize}
\end{tikzpicture}}
\newcommand{\grtroisdeux}{\begin{tikzpicture}[line cap=round,line join=round,>=triangle 45,x=0.5cm,y=0.5cm]
\clip(-0.5,-0.1) rectangle (0.5,0.7);
\draw [line width=0.5pt] (0.,0.)-- (-0.3,0.5);
\draw [line width=0.5pt] (0.,0.)-- (0.3,0.5);
\begin{scriptsize}
\draw [fill=black] (-0.3,0.5) circle (1pt);
\draw [fill=black] (0.,0.) circle (1pt);
\draw [fill=black] (0.3,0.5) circle (1pt);
\end{scriptsize}
\end{tikzpicture}}

\newcommand{\grquatreun}{\begin{tikzpicture}[line cap=round,line join=round,>=triangle 45,x=0.5cm,y=0.5cm]
\clip(-0.2,-0.1) rectangle (0.7,0.7);
\begin{scriptsize}
\draw [fill=black] (0.,0.) circle (1pt);
\draw [fill=black] (0.,0.5) circle (1pt);
\draw [fill=black] (0.5,0.5) circle (1pt);
\draw [line width=.5pt] (0.,0.)-- (0.5,0.);
\draw [line width=0.5pt] (0.,0.)-- (0.,0.5);
\draw [line width=0.5pt] (0.5,0.)-- (0.5,0.5);
\draw [line width=0.5pt] (0.0,0.)-- (0.5,0.5);
\draw [line width=0.5pt] (0.0,0.5)-- (0.5,0.5);
\draw [line width=0.5pt] (0.0,0.5)-- (0.5,0.);
\draw [fill=black] (0.5,0.) circle (1pt);
\end{scriptsize}
\end{tikzpicture}}
\newcommand{\grquatredeux}{\begin{tikzpicture}[line cap=round,line join=round,>=triangle 45,x=0.5cm,y=0.5cm]
\clip(-0.2,-0.1) rectangle (0.7,0.7);
\begin{scriptsize}
\draw [fill=black] (0.,0.) circle (1pt);
\draw [fill=black] (0.,0.5) circle (1pt);
\draw [fill=black] (0.5,0.5) circle (1pt);
\draw [line width=.5pt] (0.,0.)-- (0.5,0.);
\draw [line width=0.5pt] (0.,0.)-- (0.,0.5);
\draw [line width=0.5pt] (0.5,0.)-- (0.5,0.5);
\draw [line width=0.5pt] (0.,0.)-- (0.5,0.5);
\draw [line width=0.5pt] (0.,0.5)-- (0.5,0.5);
\draw [fill=black] (0.5,0.) circle (1pt);
\end{scriptsize}
\end{tikzpicture}}
\newcommand{\grquatretrois}{\begin{tikzpicture}[line cap=round,line join=round,>=triangle 45,x=0.5cm,y=0.5cm]
\clip(-0.2,-0.1) rectangle (0.7,0.7);
\begin{scriptsize}
\draw [fill=black] (0.,0.) circle (1pt);
\draw [fill=black] (0.,0.5) circle (1pt);
\draw [fill=black] (0.5,0.5) circle (1pt);
\draw [line width=.5pt] (0.,0.)-- (0.5,0.);
\draw [line width=0.5pt] (0.,0.)-- (0.,0.5);
\draw [line width=0.5pt] (0.5,0.)-- (0.5,0.5);
\draw [line width=0.5pt] (0.,0.)-- (0.5,0.5);
\draw [fill=black] (0.5,0.) circle (1pt);
\end{scriptsize}
\end{tikzpicture}}
\newcommand{\grquatrequatre}{\begin{tikzpicture}[line cap=round,line join=round,>=triangle 45,x=0.5cm,y=0.5cm]
\clip(-0.2,-0.1) rectangle (0.7,0.7);
\begin{scriptsize}
\draw [fill=black] (0.,0.) circle (1pt);
\draw [fill=black] (0.,0.5) circle (1pt);
\draw [fill=black] (0.5,0.5) circle (1pt);
\draw [line width=.5pt] (0.,0.)-- (0.5,0.);
\draw [line width=0.5pt] (0.,0.)-- (0.,0.5);
\draw [line width=0.5pt] (0.5,0.)-- (0.5,0.5);
\draw [line width=0.5pt] (0.,0.5)-- (0.5,0.5);
\draw [fill=black] (0.5,0.) circle (1pt);
\end{scriptsize}
\end{tikzpicture}}
\newcommand{\grquatrecinq}{\begin{tikzpicture}[line cap=round,line join=round,>=triangle 45,x=0.5cm,y=0.5cm]
\clip(-0.2,-0.1) rectangle (0.7,0.7);
\begin{scriptsize}
\draw [fill=black] (0.,0.) circle (1pt);
\draw [fill=black] (0.,0.5) circle (1pt);
\draw [fill=black] (0.5,0.5) circle (1pt);
\draw [line width=.5pt] (0.,0.)-- (0.5,0.);
\draw [line width=0.5pt] (0.,0.)-- (0.,0.5);
\draw [line width=0.5pt] (0.,0.)-- (0.5,0.5);
\draw [fill=black] (0.5,0.) circle (1pt);
\end{scriptsize}
\end{tikzpicture}}
\newcommand{\grquatresix}{\begin{tikzpicture}[line cap=round,line join=round,>=triangle 45,x=0.5cm,y=0.5cm]
\clip(-0.2,-0.1) rectangle (0.7,0.7);
\begin{scriptsize}
\draw [fill=black] (0.,0.) circle (1pt);
\draw [fill=black] (0.,0.5) circle (1pt);
\draw [fill=black] (0.5,0.5) circle (1pt);
\draw [line width=.5pt] (0.,0.)-- (0.5,0.);
\draw [line width=0.5pt] (0.,0.)-- (0.,0.5);
\draw [line width=0.5pt] (0.5,0.)-- (0.5,0.5);
\draw [fill=black] (0.5,0.) circle (1pt);
\end{scriptsize}
\end{tikzpicture}}

\input{xy}
\xyoption{all}

\definecolor{vert}{rgb}{0.,0.5,0.}

\title{Bialgebras in cointeraction, the antipode and the eulerian idempotent}
\date{}
\author{Lo\"\i c Foissy}
\affil{\small{Univ. Littoral Côte d'Opale, UR 2597
LMPA, Laboratoire de Mathématiques Pures et Appliquées Joseph Liouville
F-62100 Calais, France}.\\ Email: \texttt{foissy@univ-littoral.fr}}

\theoremstyle{plain}
\newtheorem{theo}{Theorem}[section]
\newtheorem{lemma}[theo]{Lemma}
\newtheorem{cor}[theo]{Corollary}
\newtheorem{prop}[theo]{Proposition}
\newtheorem{defi}[theo]{Definition}

\theoremstyle{remark}
\newtheorem{remark}{Remark}[section]
\newtheorem{notation}{Notations}[section]
\newtheorem{example}{Example}[section]

\newcommand{\K}{\mathbb{K}}
\newcommand{\N}{\mathbb{N}}

\renewcommand{\geq}{\geqslant}
\renewcommand{\leq}{\leqslant}

\DeclareMathOperator{\tdelta}{\tilde{\Delta}}
\DeclareMathOperator{\im}{\mathrm{Im}}

\newcommand{\id}{\mathrm{Id}}

\newcommand{\QSym}{\mathbf{QSym}}

\newcommand{\homo}{\mathrm{Hom}}
\newcommand{\en}{\mathrm{End}}
\newcommand{\chara}{\mathrm{Char}}
\newcommand{\infchara}{\mathrm{InfChar}}
\newcommand{\prim}{\mathrm{Prim}}
\newcommand{\val}{\mathrm{val}}
\newcommand{\ev}{\mathrm{ev}}
\newcommand{\eq}{\mathcal{E}}
\newcommand{\gr}{\mathcal{G}}
\newcommand{\calH}{\mathcal{H}}
\newcommand{\PVC}{\mathbf{PVC}}
\newcommand{\cl}{\mathrm{cl}}

\newcommand{\supp}{\mathrm{supp}}
\renewcommand{\ker}{\mathrm{Ker}}
\newcommand{\qsh}{\mathbf{QSh}}

\begin{document}

\maketitle

\begin{abstract}
We give here a review of results about double bialgebras, that is to say bialgebras with two coproducts,
the first one being a comodule morphism for the coaction induced by the second one. 
An accent is put on the case of connected bialgebras. 
The subjects of these results are the monoid of characters and their actions, polynomial invariants,
the antipode and the eulerian idempotent. As examples, they are applied on a double bialgebra of graphs
and on quasishuffle bialgebras. This includes a new proof of a combinatorial interpretation of the coefficients
of the chromatic polynomial due to Greene and Zaslavsky. 
\end{abstract}

\\

\AMSclass{16T05 16T30 68R15 05A05 05C15}

\tableofcontents

\section*{Introduction}

Quite recently, various combinatorial Hopf algebras equipped with a second coproduct appear in the literature:
some based on trees \cite{Calaque2011},  on several families of graphs \cite{Manchon2012,Foissy36},
on  finite topologies or posets \cite{Ayadi2020,Foissy37,Foissy34},  noncrossing partitions \cite{Ebrahimi-Fard2020},
or on words  related to Ecalle's mould calculus \cite{Ebrahimi-Fard2017-2}, for example.  
These objects play an important role in Bruned, Hairer and Zambotti's study of stochastic PDEs 
\cite{Bruned2015,Bruned2019}. 
Let us give some common properties of these objects. These are families $(B,m,\Delta,\delta)$ such that:
\begin{itemize}
\item $(B,m,\Delta)$ is a bialgebra. In most cases, it is a graded and connected Hopf algebra.
\item $(B,m,\delta)$ is a bialgebra, sharing the same product as $(B,m,\Delta)$. It is generally not a connected
coalgebra, as it contains non trivial group-like elements. Moreover, as these elements are not invertible,
this is generally not a Hopf algebra.
\item It turns out that $(B,m,\Delta)$ is a bialgebra in the category of right comodules of $(B,m,\delta)$,
with the right coaction given by $\delta$ itself: the product $m$, the coproduct $\Delta$,
the unit map $\nu$ and the counit $\varepsilon_\Delta$ of $\Delta$ are comodule morphisms.
It is rather trivial for $m$ and $\nu$, but gives the two interesting following relations for $\Delta$
and its counit $\varepsilon_\Delta$:
\begin{align*}
(\Delta \otimes \id)&=m_{1,3,24}\circ (\delta \otimes \delta)\circ \Delta,&
(\varepsilon_\Delta \otimes \id)\circ \delta&=\nu \circ \varepsilon_\Delta,
\end{align*}
where $m_{1,3,24}:B^{\otimes 4}\longrightarrow B^{\otimes 3}$ send the tensor $a_1\otimes a_2\otimes a_3\otimes a_4$
to $a_1\otimes a_3\otimes a_2a_4$. 
\end{itemize}
In other words, for such an object, $(B,m,\Delta)$ is a right-comodule bialgebra over $(B,m,\delta)$,
that is to say a bialgebra in the symmetric monoidal category of right comodules over $(B,m,\delta)$. 
For the sake of simplicity, these objects will be called double bialgebras in this text.
Considering the associated characters monoids, we obtain two products $*$ and $\star$ on the same set $\chara(B)$,
coming respectively from $\Delta$ and $\delta$, such that:
\begin{itemize}
\item $(\chara(B),*)$ is a monoid (in most cases a group).
\item $(\chara(B),\star)$ is a monoid.
\item $(\chara(B),\star)$ acts (on the right) on $(\chara(B),*)$ by monoid endomorphisms:
for any $\lambda_1$, $\lambda_2$ and $\mu \in \chara(B)$,
\[(\lambda_1*\lambda_2)\star \mu=(\lambda_1\star \mu)*(\lambda_2\star \mu).\]
\end{itemize}
In the particular case where $\Delta$ and $\delta$ are cocommutative, 
we obtain that $(\chara(B),*,\star)$ is in fact a ring. \\

Our aim in this text is a review of the theoretical consequences of this setting, illustrated by examples
based on words with quasishuffle products and on graphs, with an unexpected application to the eulerian idempotent.
We start with general results, with no particular hypothesis on the structure of $(B,m,\Delta)$. 
We show that, as mentioned before, the monoid of characters $(\chara(B),\star)$ of $(B,m,\delta)$
acts on the monoid of characters $(\chara(B),*)$ of $(B,m,\Delta)$, but also on the space $\homo(B,V)$
of linear homomorphisms from $B$ to any vector space $V$ (Proposition \ref{prop2.5}). 
If $V$ is an algebra (respectively a bialgebra or a coalgebra), the subset of algebra (respectively bialgebra or coalgebra)
morphisms is stable under this action. We also prove that, in the case where $(B,m,\Delta)$ is a Hopf algebra,
then its antipode $S$ is  automatically a comodule morphism (Proposition \ref{prop2.1}), that is to say:
\[\delta \circ S=(S\otimes \id)\circ \delta.\]
We also introduce an important tool, the map $\Theta$, which sends a linear form $\lambda$ on $B$ to
the endomorphism $\Theta(\lambda)=(\lambda \otimes \id)\circ \delta$. 
We prove in Proposition \ref{prop2.2} that this map is compatible with both $*$ and $\star$:
for any $\lambda,\mu\in B^*$,
\begin{align*}
\Theta(\lambda*\mu)&=\Theta(\lambda)*\Theta(\mu),&
\Theta(\lambda\star \mu)&=\Theta(\mu)\circ \Theta(\lambda).
\end{align*}
As an example of consequence, we give in Corollary \ref{cor2.3} a criterion for the existence of the antipode
for $(B,m,\Delta)$: this is a Hopf algebra, if and only if, the counit $\epsilon_\delta$ of the coproduct $\delta$ is invertible 
for the convolution product $*$ dual of $\Delta$, and then the antipode is $\Theta(\epsilon_\delta^{*-1})$.
An immediate consequence is that $S$ is an algebra morphism -- and an algebra antimorphism, by a very classical result.
Consequently,  we obtain that $(H,m)$ is commutative. By the way, this explains why no non commutative
example of double bialgebra was known. \\

We then add the assumption that $(B,\Delta)$ is a connected coalgebra. This gives the existence
of an increasing filtration $(B_{\leq n})_{n\in \N}$ (the coradical filtration) of $B$ such that for any $k,l,n\in \N$,
\begin{align*}
m(B_{\leq k}\otimes B_{\leq l})&\subseteq B_{\leq k+l},&
\Delta(B_{\leq n})&\subseteq \sum_{p=0}^n B_{\leq p}\otimes B_{\leq n-p},
\end{align*}
and such that $B_{\leq 0}=\K 1_B$. In this case, for any vector space $V$, $\en(B,V)$ inherits a distance, making it
a complete hypermetric space. When $V$ is an algebra, then $\en(B,V)$ inherits a convolution product,
which makes it a complete hypermetric algebra. Moreover, if $f\in \en(B,V)$ satisfies $f(1_B)=0$,
we obtain a continuous algebra map from the algebra of formal series $\K[[T]]$ to $\en(B,V)$,
which sends $T$ to $f$ (Proposition \ref{prop3.3}): 
this allows to define exponential, logarithm, or non integral powers of elements of $\en(B,V)$.
This formalism can be used to prove Takeuchi's formula for the antipode, a universal property for shuffle coalgebras
(Proposition \ref{prop3.5}), or the well-known exp-ln bijection between the Lie algebra of infinitesimal characters
to the  group of characters of $(B,m,\Delta)$ (Proposition \ref{prop3.6}). 
 
One of the simplest examples of double bialgebra is the polynomial algebra  $\K[X]$, with its two coproducts defined by
\begin{align*}
\Delta(X)&=X\otimes 1+1\otimes X,&\delta(X)&=X\otimes X.
\end{align*}
We prove in Theorem \ref{theo3.9} that it is a terminal object in the category of connected double bialgebras:
in other words, for any connected double bialgebra $(B,m,\Delta,\delta)$, there exists a unique double bialgebra morphism
from $B$ to $\K[X]$. Moreover, this morphism is
\[\Phi=\epsilon_\delta^X=(1+(\epsilon_\delta-\varepsilon_\Delta))^X
=\sum_{n=0}^\infty \dfrac{X(X-1)\ldots (X-n+1)}{n!} \epsilon_\delta^{\otimes n}\circ \tdelta^{(n-1)},\]
 with the use of formal series described earlier, and where the maps $\tdelta^{(n-1)}$ are the iterated of the reduced
 coproduct $\tdelta$.  We also prove that this morphism $\Phi$ allows to construct all bialgebra morphisms 
 from $(B,m,\Delta)$ to $(\K[X],m,\Delta)$, thanks to the action of the monoid of characters $(\chara(B),\star)$, 
see Corollary \ref{theo3.12}. When applied to the double bialgebra of graphs, this gives
the chromatic polynomial (Theorem \ref{theo3.13}). When applied to a quasishuffle double bialgebra,
this gives a morphism involving Hilbert polynomials (Proposition \ref{prop3.14}). \\

When one works with the enveloping algebra $\mathcal{U}(\mathfrak{g})$ of a Lie algebra $\mathfrak{g}$, 
the eulerian idempotent is a useful projector on $\mathfrak{g}$, see \cite{Loday1994,Burgunder2008,
Bandiera2017} for several applications. It is originally defined on the enveloping algebra of a free Lie algebra,
in terms of descents of permutations \cite{Solomon1968}.
This can be generalized without any problem to any connected bialgebra, by the formula
\[\varpi=\ln(\id)=\sum_{k=1}^\infty \dfrac{(-1)^{k+1}}{k}m^{(k-1)}\circ \tdelta^{(k-1)},\]
where the $m^{(k-1)}$ are the iterated products. It is generally not a projector. If $B$ is cocommutative, 
it is well-known that it is a projector on the Lie algebra of the primitive elements of $B$. 
The case of a commutative connected bialgebra is not so well known. 
We here consider the case of a connected double bialgebra $B$. 
An especially interesting infinitesimal character is given by the logarithm $\phi$ of the counit $\epsilon_\delta$. 
We prove that:
\begin{itemize}
\item $\phi$ is closely related to the double bialgebra morphism $\Phi$, see Proposition \ref{prop4.1}: for any $x\in B$,
\[\phi(x)=\Phi(x)'(0).\]
\item the eulerian idempotent of $B$ is $\Theta(\phi)=(\phi\otimes \id)\circ \delta$, see  Proposition \ref{prop4.1}.
\item for any character $\lambda$ of $B$, $\ln(\lambda)=\phi\star \lambda$, see Proposition \ref{prop4.2}.
\item  for any infinitesimal character $\mu$ of $B$, $\phi\star\mu=\mu$, see Lemma \ref{lemma4.3}.
\end{itemize}
Consequently, $\phi\star\phi=\phi$, which implies that $\varpi$ is a projector, 
that its kernel is $\K 1_B\oplus B_+^2$ (Proposition \ref{prop4.4}),
and its image contains the Lie algebra $\prim(B)$ of primitive elements of $B$ (but is not equal, 
except if $B$ is cocommutative). This result can be extended to any commutative connected bialgebra
(Proposition \ref{prop4.13}).
In the case of the graph bialgebra, this infinitesimal character admits a combinatorial interpretation
in term of acyclic orientations with a single fixed source (Theorem \ref{theo4.9}). 
For quasishuffle bialgebras, the eulerian idempotent is given in Corollary \ref{cor4.17}, in term of descents of surjections.
Applications of this projector include that any commutative connected bialgebra can be seen as a subbialgebra
of a shuffle bialgebra (Corollary \ref{cor4.14}), which in turns implies Hoffman's result \cite{Hoffman2000}
that any commutative quasishuffle bialgebra is isomorphic to a shuffle algebra and that any
commutative connected bialgebra can be embedded in a double bialgebra. \\

We then make precise the hypothesis on $(B,m,\Delta)$ and assume that it is connected and graded:
there exists a family of subspaces $(B_n)_{n\in \N}$ of $B$ such that
\[B=\bigoplus_{n=0}^\infty B_n,\]
and such that  for any $k,l,n\in \N$,
\begin{align*}
m(B_k\otimes B_l)&\subseteq B_{k+l},&
\Delta(B_n)&\subseteq \sum_{p=0}^n B_p\otimes B_{n-p},
\end{align*}
and with $B_0=\K 1_B$. A natural question is the description of homogeneous morphisms from
$(B,m,\Delta)$ to $(\K[X],m,\Delta)$ (noting that the unique double bialgebra morphism is usually not homogeneous).
we obtain that these morphisms are in bijection with the space $B_1^*$ (Proposition \ref{prop5.2}),
with explicit formulas (Corollary \ref{cor5.3}). In the case of graphs, taking $\lambda \in B_1^*$
defined by $\lambda(\grun)=1$, we obtain the bialgebra morphism sending any graph $G$ of degree $n$
to $X^n$, and the action of $(\chara(B),\star)$ allows to recover the interpretation of the coefficients of the chromatic
polynomials in terms of acyclic orientations of \cite{Greene1983,Deb2019}. 
Finally, we also consider, following Aguiar, Bergeron and Sottile's result \cite{Aguiar2006-2}, 
that under a homogeneity condition, there exists a unique homogeneous double bialgebra morphism 
from $(B,m,\Delta,\delta)$ to the double bialgebra of quasisymmetric functions $\QSym$, 
which is a special case of a quasishuffle double bialgebra based on a semigroup \cite{Ebrahimi-Fard2017-2}.\\

This paper is organized as follows: the first section recalls the definition of double bialgebras,
the examples of graphs and of quasishuffle bialgebras.  The second section gives general results on double bialgebras,
including the properties of the map $\Theta$ and the actions of the monoid of characters. 
The third part concentrates on the particular case of connected double bialgebras,
with the exp-ln bijection between infinitesimal characters and characters and the polynomial invariants.
The eulerian projector, its properties and their consequences, are studied in the next section, 
and the last section gives results in the more specific case of graded double bialgebras. \\

\textbf{Acknowledgements}. 
The author acknowledges support from the grant ANR-20-CE40-0007
\emph{Combinatoire Algébrique, Résurgence, Probabilités Libres et Opérades}.\\

\begin{notation}
\begin{itemize}
\item $\K$ is a commutative field of characteristic zero. All the vector spaces in this text will be taken over $\K$.
\item For any $k\in \N$, we put $[k]=\{1,\ldots,k\}$. In particular, $[0]=\emptyset$. 
\item We denote by $\K[[T]]$ the algebra of formal series with coefficients in $\K$.
If $P(T)=\sum a_nT^n\in \K[[T]]$, the valuation of $P(T)$ is
\[\val(P(T))=\min\{n\in \N\mid a_n \neq 0\}.\]
By convention, $\val(0)=+\infty$. This induces a distance $d$ on $\K[[T]]$, defined by
\[d(P(T),Q(T))=2^{-\val(P(T)-Q(T))},\]
with the convention $2^{-\infty}=0$. Then $(\K[[T]],d)$ is a complete metric space.
If $P(T)=\sum a_nT^n$ and $Q(T)\in \K[[T]]$ with $Q(0)=0$, the composition of $P$ and $Q$ is
\[P\circ Q(T)=\sum_{n=0}^\infty a_n Q(T)^n.\]
We shall use the classical formal series
\begin{align*}
e^T&=\sum_{n=0}^\infty \frac{T^n}{n!},\\
\ln(1+T)&=\sum_{n=1}^\infty \frac{(-1)^{n+1}}{n}T^n,\\
(1+T)^x&=\sum_{n=0}^\infty \frac{x(x-1)\ldots (x-n+1)}{n!}T^n,&\mbox{for $x\in \K$}. 
\end{align*}
We recall the classical results:
\begin{align*}
\left(e^T-1\right)\circ \ln(1+T)&=\ln(1+T)\circ \left(e^T-1\right)=T,&
(1+T)^x&=e^{x\ln(1+T)}.
\end{align*}
Moreover, for any formal series $P(T)$ and $Q(T)$ with no constant terms,
\begin{align*}
e^T\circ (P(T)+Q(T))&=\left(e^T\circ P(T)\right)\left(e^T\circ Q(T)\right),\\
\ln(1+T)\circ (P(T)+Q(T)+P(T)Q(T))&=\ln(1+T)\circ P(T)+\ln(1+T)\circ Q(T),
\end{align*}
which implies that for any $x,y\in \K$,
\[(1+T)^{x+y}=(1+T)^x(1+T)^y.\]
\end{itemize}
\end{notation}

\section{Cointeracting bialgebras}

\subsection{Definition}

We refer to the references \cite{Abe1980,Cartier2021,Sweedler1969} for the main definitions on bialgebras
and Hopf algebras. 
Let $(B,m,\delta)$ be a bialgebra. Its counit will be denoted by $\epsilon_\delta$.
It is well-known that its category of (right) comodules is a monoidal category:
\begin{itemize}
\item If $(M_1,\rho_1)$ and $(M_2,\rho_2)$ are two comodules over $B$, then $M_1\otimes M_2$
is also a comodule, with the coaction $m_{1,3,24}\circ (\rho_1\otimes \rho_2)$,
with
\[m_{1,3,24}:\left\{\begin{array}{rcl}
M_1\otimes B\otimes M_2\otimes B&\longrightarrow&M_1\otimes M_2\otimes B\\
m_1\otimes b_1\otimes m_2\otimes b_2&\longrightarrow&m_1\otimes m_2\otimes b_1b_2.
\end{array}\right.\]
\item If $f_1:M_1\longrightarrow M_1'$ and $f_2:M_2\longrightarrow M_2'$ are comodule morphisms,
then $f_1\otimes f_2:M_1\otimes M_2\longrightarrow M_1'\otimes M_2'$ is a comodule morphism. 
\item The associativity of $m$ implies that if $(M_1,\rho_1)$, $(M_2,\rho_2)$ and $(M_3,\rho_3)$
are three comodules over $B$, then $(M_1\otimes M_2)\otimes M_3$ and $M_1\otimes (M_2\otimes M_3)$
are the same comodule. 
\item The unit comodule is $\K$ with the coaction defined by 
\begin{align*}
&\forall x\in \K,&\rho(x)&=x\otimes 1_B.
\end{align*}
The canonical identifications of $\K\otimes M$ and $M\otimes \K$ with $M$ are
comodules isomorphims for any comodule $M$. 
\end{itemize}

In particular, $B$ is a comodule over itself with the coaction $\delta$. Hence, for any $n\in \N$,
$B^{\otimes n}$ is a comodule over $B$, with the coaction $m_{1,3,\ldots, 2n-1,24\ldots 2n}\circ \delta^{\otimes n}$,
where
\[m_{1,3,\ldots, 2n-1,24\ldots 2n}:\left\{\begin{array}{rcl}
B^{\otimes 2n}&\longrightarrow&B^{\otimes n}\\
b_1\otimes \ldots \otimes b_{2n+1}&\longrightarrow&b_1\otimes b_3\otimes \ldots \otimes b_{2n-1}
\otimes b_2b_4\ldots b_{2n}.
\end{array}\right.\]
Note that $m:B\otimes B\longrightarrow B$ is always a comodule morphism, as well as the unit map
$\nu_B:\K\longrightarrow B$, which sends $x\in \K$ to $x1_B$. 
A double bialgebra is given by a coproduct $\Delta$ on $B$, making $(B,m,\Delta)$ a bialgebra in the category
of comodules over $(B,m,\delta)$ with the coaction $\delta$. In more details:

\begin{defi}
A \emph{double bialgebra} is a family $(B,m,\Delta,\delta)$ such that:
\begin{itemize}
\item $(B,m,\delta)$ is a bialgebra. Its counit is denoted by $\epsilon_\delta$.
\item $(B,m,\Delta)$ is a bialgebra. Its counit is denoted by $\varepsilon_\Delta$.
\item $\Delta:B\longrightarrow B\otimes B$ is a comodule morphism:
\[(\Delta \otimes \id_B)\circ \delta=m_{1,3,24}\circ (\delta \otimes \delta)\circ \Delta.\]
\item $\varepsilon_\Delta:B\longrightarrow \K$ is a comodule morphism:
\[(\varepsilon_\Delta \otimes \id)\circ \delta=\nu_B\circ \varepsilon_\Delta.\]
\end{itemize}\end{defi}

\begin{remark}
Let $(B,m,\Delta,\delta)$ be a double bialgebra. Then, as $(B,m,\delta)$ is a bialgebra,
\[\delta\circ m=m_{13,24}\circ (\delta\otimes \delta)=(m\otimes \id)\circ m_{1,3,24}\circ (\delta\otimes \delta),\]
with the obvious notation $m_{13,24}$. Therefore, $m$ is a comodule morphism from $B\otimes B$ to $B$.
Moreover, as $\delta(1_B)=1_B\otimes 1_B$, the map $\nu_B:\K\longrightarrow B$ is a comodule morphism.\\
\end{remark}

\begin{example}
The algebra $\K[X]$ is a double bialgebra, with the two multiplicative coproducts defined by
\begin{align*}
\Delta(X)&=X\otimes 1+1\otimes X,&\delta(X)&=X\otimes X.
\end{align*}
In other terms, identifying $\K[X]\otimes \K[X]$ with $\K[X,Y]$ through the algebra map
\[\left\{\begin{array}{rcl}
\K[X]\otimes \K[X]&\longrightarrow&\K[X,Y]\\
P(X)\otimes Q(X)&\longrightarrow&P(X)Q(Y),
\end{array}\right.\]
for any $P\in \K[X]$,
\begin{align*}
\Delta(P(X))&=P(X+Y),&\delta(P(X))&=P(XY).
\end{align*}
The counit $\varepsilon_\Delta$ sends $P\in \K[X]$ to $P(0)$ and the counit $\epsilon_\delta$ sends it to $P(1)$.
\end{example}

\subsection{The example of graphs}

We refer to \cite{Harary1969} for classical definitions and notations on graphs.
In the context of this article, a graph will be a pair $G=(V(G),E(G))$, where $V(G)$ is a finite set (maybe empty),
called the set of vertices, and $E(G)$ a sets of 2-elements sets of elements of $V(G)$, called the set  of edges of $G$. 
The degree of $G$ is the cardinality of $V(G)$. 
If $G$ and $H$ are two graphs, an isomorphism from $G$ to $H$ is a bijection $f:V(G)\longrightarrow V(H)$ such 
that for any $x\neq y\in V(G)$, $\{x,y\}\in E(G)$ if, and only if, $\{f(x),f(y)\}\in E(H)$.
We shall denote by $\gr$ the set of isoclasses of graphs, and for any $n\in \N$, by $\gr(n)$ the set of isoclasses of graphs
of degree $n$. The vector space generated by $\gr$ will be denoted by $\calH_\gr$. 

\begin{example} 
\begin{align*}
\gr(0)&=\{1\},\\
\gr(1)&=\{\grun\},\\
\gr(2)&=\{\grdeux,\grun\grun\},\\
\gr(3)&=\{\grtroisun,\grtroisdeux,\grdeux\grun,\grun\grun\grun\},\\
\gr(4)&=\{\grquatreun,\grquatredeux,\grquatretrois,\grquatrequatre,\grquatrecinq,\grquatresix,
\grtroisun\grun,\grtroisdeux\grun,\grdeux\grdeux,\grdeux\grun\grun,\grun\grun\grun\grun\}.
\end{align*}
\end{example}

If $G$ and $H$ are two graphs, their disjoint union is the graph $GH$ defined by 
\begin{align*}
V(GH)&=V(G)\sqcup V(H), &E(GH)&=E(G)\sqcup E(H).
\end{align*}
This induces a commutative and associative product $m$ on $\calH_\gr$,
whose units is the empty graph $1$.\\

Let $G$ be a graph and $I\subseteq V(G)$. The subgraph $G_{\mid I}$ is defined by
\begin{align*}
V(G_{\mid I})&=I,&E(G_{\mid I})&=\{\{x,y\}\in E(G)\mid x,y\in I\}.
\end{align*}
This notion induces a commutative and coassociative coproduct $\Delta$ on $\calH_\gr$ given by
\begin{align*}
&\forall G\in\gr,&\Delta(G)&=\sum_{I\subseteq V(G)}G_{\mid I}\otimes G_{\mid V(G)\setminus I}.
\end{align*}
Its counit $\varepsilon_\Delta$ is given by
\begin{align*}
&\forall G\in\gr,&\varepsilon_\Delta(G)&=\delta_{G,1}.
\end{align*}

\begin{example}
\begin{align*}
\Delta(\grun)&=\grun \otimes 1+1\otimes \grun,\\
\Delta(\grdeux)&=\grdeux\otimes 1+1\otimes \grdeux+2\grun \otimes \grun,\\
\Delta(\grtroisun)&=\grtroisun\otimes 1+1\otimes \grtroisun+3\grun \otimes \grdeux+3\grdeux\otimes \grun,\\
\Delta(\grtroisdeux)&=\grtroisdeux\otimes 1+1\otimes\grtroisdeux+2\grun \otimes \grdeux+
\grun\otimes \grun\grun+2\grdeux\otimes \grun+\grun\grun\otimes \grun,\\
\Delta(\grquatreun)&=\grquatreun\otimes 1+1\otimes\grquatreun
+4\grun \otimes \grtroisun+6\grdeux\otimes \grdeux+4\grtroisun\otimes \grun,\\
\Delta(\grquatredeux)&=\grquatredeux\otimes 1+1\otimes\grquatredeux
+2\grun \otimes \grtroisun+2\grun\otimes \grtroisdeux\\
&+4\grdeux\otimes \grdeux+\grdeux\otimes \grun\grun
+\grun\grun\otimes \grdeux+2\grtroisun\otimes \grun+2\grtroisdeux\otimes \grun,\\
\Delta(\grquatretrois)&=\grquatretrois\otimes 1+1\otimes\grquatretrois
+\grun \otimes \grtroisun+2\grun \otimes \grtroisdeux+\grun \otimes \grdeux\grun\\
&+2\grdeux\otimes \grdeux+2\grun\grun\otimes \grdeux+2\grdeux\otimes \grun\grun
+\grtroisun\otimes \grun+2\grtroisdeux\otimes \grun+\grdeux\grun\otimes \grun,\\
\Delta(\grquatrequatre)&=\grquatrequatre\otimes 1+1\otimes\grquatrequatre
+4\grun \otimes \grtroisdeux+4\grdeux\otimes \grdeux+2\grun\grun\otimes \grun\grun
+4\grtroisdeux\otimes \grun,\\
\Delta(\grquatrecinq)&=\grquatrecinq\otimes 1+1\otimes\grquatrecinq+3\grun \otimes \grtroisdeux
+3\grun\otimes \grun\grun\grun+3\grdeux\otimes \grun\grun+3\grun\grun\otimes \grdeux
+\grtroisdeux\otimes \grun+\grun\otimes \grun\grun\grun,\\
\Delta(\grquatresix)&=\grquatresix\otimes 1+1\otimes\grquatresix
+2\grun\otimes \grtroisdeux+2\grun\otimes\grdeux\grun\\
&+2\grdeux\otimes\grdeux+2\grun\grun\otimes\grun\grun+\grun\grun\otimes\grdeux
+\grdeux\otimes\grun\grun+2\grtroisdeux\otimes \grun+2\grdeux\grun\otimes \grun.
\end{align*}
\end{example}

Let $G$ be a graph and let $\sim$ be an equivalence relation on $V(G)$. We define the contracted graph $G/\sim$ by
\begin{align*}
V(G/\sim)&=V(G)/\sim,&E(G/\sim)&=\{\{\pi_\sim(x),\pi_\sim(y)\}\mid \{x,y\}\in E(G), \:\pi_\sim(x)\neq \pi_\sim(y)\},
\end{align*}
where $\pi_\sim:V(G)\longrightarrow V(G)/\sim$ is the canonical surjection. 
We define the restricted graph $G\mid\sim$ by
\begin{align*}
V(G\mid \sim)&=V(G),&E(G\mid \sim)=\{\{x,y\}\in E(G)\mid \pi_\sim(x)=\pi_\sim(y)\}.
\end{align*}
In other words, $G\mid \sim$ is the disjoint union of the subgraphs $G_{\mid C}$, with $C\in V(G)/\sim$.
We shall say that $\sim\in\eq_c(G)$ if for any class $C\in V(G)/\sim$, $G_{\mid C}$ is a connected graph.
We then define a second coproduct $\delta$ on $\calH_\gr$ by
\begin{align*}
&\forall G\in\gr,&\delta(G)&=\sum_{\sim\in \eq_c(G)} G/\sim\otimes G\mid \sim.
\end{align*}
This coproduct is coassociative, but not cocommutative. Its counit $\epsilon_\delta$ is given by
\begin{align*}
&\forall G\in\gr,&\epsilon_\delta(G)&=\begin{cases}
1\mbox{ if }E(G)=\emptyset,\\
0\mbox{ otherwise}.
\end{cases}
\end{align*}

\begin{example}
\begin{align*}
\delta(\grun)&=\grun \otimes\grun,\\
\delta(\grdeux)&=\grdeux\otimes \grun\grun+ \grun\otimes \grdeux,\\
\delta(\grtroisun)&=\grtroisun\otimes \grun\grun\grun+ \grun\otimes \grtroisun
+3\grdeux\otimes \grdeux\grun,\\
\delta(\grtroisdeux)&=\grtroisdeux\otimes\grun\grun\grun+2\grdeux\otimes \grdeux\grun,\\
\delta(\grquatreun)&=\grquatreun\otimes\grun\grun\grun\grun+ \grun\otimes\grquatreun
+6\grtroisun \otimes \grdeux\grun\grun+\grdeux\otimes (6\grdeux\grdeux+4\grtroisun\grun),\\
\delta(\grquatredeux)&=\grquatredeux\otimes\grun\grun\grun\grun+ \grun\otimes\grquatredeux
+(4\grtroisun+\grtroisdeux)\otimes \grdeux\grun\grun
+\grdeux\otimes (2\grdeux\grdeux+2\grun\grtroisun+2\grun\grtroisdeux),\\
\delta(\grquatretrois)&=\grquatretrois\otimes\grun\grun\grun\grun+ \grun\otimes\grquatretrois
+(\grtroisun+3\grtroisdeux)\otimes \grdeux\grun\grun
+\grdeux\otimes (\grdeux\grdeux+\grun\grtroisun+2\grun\grtroisdeux),\\
\delta(\grquatrequatre)&=\grquatrequatre\otimes\grun\grun\grun\grun+ \grun\otimes\grquatrequatre
+4\grtroisun\otimes \grdeux\grun\grun +\grdeux\otimes (2\grdeux\grdeux+4\grun\grtroisun),\\
\delta(\grquatrecinq)&=\grquatrecinq\otimes\grun\grun\grun\grun+ \grun\otimes\grquatrecinq
+3\grtroisdeux\otimes \grdeux\grun\grun+3\grdeux\otimes \grun\grtroisdeux,\\
\delta(\grquatresix)&=\grquatresix\otimes\grun\grun\grun\grun+ \grun\otimes\grquatresix
+3\grtroisdeux\otimes \grdeux\grun\grun +\grdeux\otimes (\grdeux\grdeux+2\grun\grtroisdeux).
\end{align*}
\end{example}

\begin{prop} \cite{Foissy36}
$(\calH_\gr,m,\Delta,\delta)$ is a double bialgebra. 
\end{prop}

\subsection{Quasishuffle algebras}

\begin{notation}
Let $k,l\in \N$. A $(k,l)$-quasishuffle is a surjective map $\sigma:[k+l]\longrightarrow [\max(\sigma)]$ such that 
$\sigma(1)<\ldots<\sigma(k)$ and $\sigma(k+1)<\ldots<\sigma(k+l)$. 
The set of $(k,l)$-quasishuffles is denoted by $\qsh(k,l)$. 
\end{notation}

Let $(V,\cdot)$ be a commutative algebra (not necessarily unitary). The quasishuffle bialgebra associated to $V$
is $(T(V),\squplus,\Delta)$, where
\begin{align*}
&\forall v_1,\ldots,v_{k+l}\in V,&v_1\ldots v_k\squplus v_{k+1}\ldots v_{k+l}
&=\sum_{\sigma \in \qsh(k,l)} \left(\prod_{\sigma(i)=1}^\cdot v_i\right)\ldots
\left(\prod_{\sigma(i)=\max(\sigma)}^\cdot v_i\right),
\end{align*}
where the symbol $\displaystyle \prod^\cdot$ means that the products are taken in $(V,\cdot)$. 
For example, if $v_1,v_2,v_3,v_4\in V$,
\begin{align*}
v_1\squplus v_2v_3v_4&=v_1v_2v_3v_4+v_2v_1v_3v_4+v_2v_3v_1v_4+v_2v_3v_4v_1\\
&+(v_1\cdot v_2)v_3v_4+v_2(v_1\cdot v_3)v_4+v_2v_3(v_1\cdot v_4),\\
v_1v_2\squplus v_3v_4&=v_1v_2v_3v_4+v_1v_3v_2v_4+v_1v_3v_4v_2+v_3v_1v_2v_4+v_3v_1v_4v_2+v_3v_4v_1v_2\\
&+(v_1\cdot v_3)v_2v_4+(v_1\cdot v_3)v_2v_4+v_3(v_1\cdot v_4)v_2\\
&+v_1(v_2\cdot v_3)v_4+v_1v_3(v_2\cdot v_4)+v_3v_1(v_2\cdot v_4)+(v_1\cdot v_3)(v_2\cdot v_4).
\end{align*}
The coproduct is the deconcatenation coproduct:
\begin{align*}
&\forall v_1,\ldots,v_n\in V,&\Delta(v_1\ldots v_n)&=\sum_{k=0}^n v_1\ldots v_k\otimes v_{k+1}\ldots v_n.
\end{align*}
In the particular case where $\cdot=0$, we obtain the quasishuffle algebra $(T(V),\shuffle,\Delta)$.\\

When $(V,\cdot,\delta_V)$ is a commutative (not necessarily unitary) bialgebra, 
then $(T(V),\squplus,\Delta)$ inherits a second coproduct $\delta$:
\begin{align*}
&\forall v_1,\ldots,v_n\in V,&&\delta(v_1\ldots v_n)\\
&&&=\sum_{1\leq i_1<\ldots<i_k<n}
\left(\prod_{0< i\leq i_1}^\cdot v'_i\right)\ldots \left(\prod_{i_k< i\leq n}^\cdot v'_i\right)
\otimes v''_1\ldots v''_{i_1}\squplus \ldots \squplus v''_{i_k+1}\ldots v''_n,
\end{align*}
with Sweedler's notation $\delta_V(v)=v'\otimes v''$ for any $v\in V$. For example, if $v_1,v_2,v_3\in V$,
\begin{align*}
\delta(v_1)&=v'_1\otimes v_1'',\\
\delta(v_1v_2)&=v_1'v_2'\otimes v_1''\squplus v_2''+v_1'\cdot v_2'\otimes v_1''v_2'',\\
\delta(v_1v_2v_3)&=v_1'v_2'v_3'\otimes v_1''\squplus v_2''\squplus v_3''
+(v_1'\cdot v_2')v_3'\otimes v_1'' v_2''\squplus v_3''\\
&+v_1'(v_2'\cdot v_3')\otimes v_1''\squplus v_2'' v_3''+(v_1'\cdot v_2'\cdot v_3')\otimes v_1'' v_2'' v_3''.
\end{align*}

\begin{prop}
If $(V,\cdot,\delta_V)$ is a commutative (not necessarily unitary) bialgebra, then $(T(V),\squplus,\Delta,\delta)$ is a double bialgebra.
\end{prop}

\begin{proof}
It is quite well-known that $(T(V),\squplus,\Delta)$ is a bialgebra \cite{Hoffman2000,Hoffman2020}.
We shall use the following notation: for any $w=v_1\ldots v_n \in V^{\otimes n}$, with $n\geq 1$,
\begin{align*}
|w|&=v_1\cdot \ldots \cdot v_n,\\
w'\otimes w''&=v'_1\ldots v'_n\otimes v''_1\ldots v''_n,
\end{align*}
where we used Sweedler's notation $\delta_V(v)=v'\otimes v''$ for any $v\in V$. 
Let $w\in V^{\otimes n}$, with $n\geq 1$. Then
\[\delta(w)=\sum_{\substack{w=w_1\ldots w_n,\\ w_1,\ldots,w_n\neq 1}}
|w'_1|\ldots |w'_n|\otimes w''_1\squplus \ldots \squplus w''_n.\]

\textit{First step}. Let  $w\in V^{\otimes n}$, with $n\geq 1$. 
\begin{align*}
(\Delta \otimes \id)\circ \delta(w)&=\sum_{\substack{w=w_1\ldots w_{k+l},\\ w_1,\ldots,w_{k+l}\neq 1}}
|w'_1|\ldots |w'_k|\otimes |w'_{k+1}|\ldots |w''_{k+l}|\otimes w''_1\squplus \ldots \squplus w''_n\\
&=\sum_{\substack{w=w^{(1)}w^{(2)},\\w^{(1)}=w_1\ldots w_k,\\ w_1,\ldots,w_k\neq 1,\\
w^{(2)}=w_{k+1}\ldots w_{k+l},\\ w_{k+1},\ldots,w_{k+l}\neq 1}}
|w'_1|\ldots |w'_k|\otimes |w'_{k+1}|\ldots |w''_{k+l}|\otimes w''_1\squplus \ldots \squplus w''_n\\
&=\squplus_{1,3,24}\circ (\delta \otimes \delta)\circ \Delta(w).
\end{align*}

\textit{Second step}. Let us prove that for any $x\in V^{\otimes k}$, $y\in V^{\otimes l}$,
\[\squplus_{13,24}\circ (\delta \otimes \delta)(x\otimes y)=\delta \circ \squplus(x\otimes y).\]
We proceed by induction on $n=k+l$. If $k=0$, we can assume that $x=1$ and then
\[\squplus_{13,24}\circ (\delta \otimes \delta)(1\otimes y)=\delta(y)
=\delta \circ \squplus(x\otimes y).\]
The result also holds if $l=0$: these observations give the cases $n=0$ and $n=1$. 
Let us now assume that $k,l\geq 1$ and the result at all ranks $<n$. 
\begin{align*}
(\Delta\otimes \id)\circ \squplus_{13,24}\circ (\delta \otimes \delta)(x\otimes y)
&=\squplus_{14,25,36}\circ (\Delta \otimes \id \otimes \Delta \otimes \id)\circ (\delta \otimes \delta)(x\otimes y)\\
&=\squplus_{14,25,36}\circ \squplus_{1,3,24,5,7,68}\circ (\delta \otimes \delta \otimes \delta \otimes \delta)
\circ (\Delta \otimes \Delta)(x\otimes y)\\
&=\squplus_{15,37,2468}\circ (\delta \otimes \delta \otimes \delta \otimes \delta)
\circ (\Delta \otimes \Delta)(x\otimes y),
\end{align*}
whereas, with Sweedler's notation $\delta(z)=z^{(1)}\otimes z^{(2)}$ for any $z\in T(V)$,
\begin{align*}
&(\Delta \otimes \id)\circ \delta \circ \squplus(x\otimes y)\\
&=\squplus_{1,3,24}\circ (\delta \otimes \delta)\circ \Delta \circ \squplus(x\otimes y)\\
&=\squplus_{1,3,24}\circ (\delta \otimes \delta)\circ \squplus_{13,24}\circ (\Delta \otimes \Delta)(x\otimes y)\\
&=\squplus_{1,3,24}\circ (\delta \otimes \delta)\circ \squplus_{13,24}\left(\Delta \otimes \Delta(x\otimes y)
-x\otimes 1\otimes y\otimes 1-1\otimes x\otimes 1\otimes y\right)\\
&+(x\squplus y)^{(1)}\otimes 1\otimes (x\squplus y)^{(2)}+1\otimes (x\squplus y)^{(1)}\otimes (x\squplus y)^{(2)}\\
&=\squplus_{1,3,24}\circ \squplus_{15,24,37,68}\circ (\delta \otimes \delta \otimes \delta \otimes \delta )
\left(\Delta \otimes \Delta(x\otimes y)-x\otimes 1\otimes y\otimes 1-1\otimes x\otimes 1\otimes y\right)\\
&+(x\squplus y)^{(1)}\otimes 1\otimes (x\squplus y)^{(2)}
+1\otimes (x\squplus y)^{(1)}\otimes (x\squplus y)^{(2)}\\
&=\squplus_{15,37,2468}\circ (\delta \otimes \delta \otimes \delta \otimes \delta)
\circ (\Delta \otimes \Delta)(x\otimes y)\\
&-x^{(1)}\squplus y^{(1)}\otimes 1\otimes x^{(2)}\squplus y^{(2)}
-1\otimes x^{(1)}\squplus y^{(1)}\otimes x^{(2)}\squplus y^{(2)}\\
&+(x\squplus y)^{(1)}\otimes 1\otimes (x\squplus y)^{(2)}
+1\otimes (x\squplus y)^{(1)}\otimes (x\squplus y)^{(2)}\\
&=(\Delta\otimes \id)\circ \squplus_{13,24}\circ (\delta \otimes \delta)(x\otimes y)\\
&-x^{(1)}\squplus y^{(1)}\otimes 1\otimes x^{(2)}\squplus y^{(2)}
-1\otimes x^{(1)}\squplus y^{(1)}\otimes x^{(2)}\squplus y^{(2)}\\
&+(x\squplus y)^{(1)}\otimes 1\otimes (x\squplus y)^{(2)}
+1\otimes (x\squplus y)^{(1)}\otimes (x\squplus y)^{(2)}.
\end{align*}
We use the induction hypothesis for the fourth equality. We obtain that
\[(\tdelta \otimes \id\otimes \id)\circ \delta \circ \squplus(x\otimes y)
=(\tdelta\otimes \id)\circ \squplus_{13,24}\circ (\delta \otimes \delta)(x\otimes y),\]
so \[\delta \circ \squplus(x\otimes y)-\squplus_{13,24}\circ (\delta \otimes \delta)(x\otimes y)\in V\otimes T(V).\]
Let $\pi$ be the canonical projection from $T(V)$ onto $V$. For any $w\in V^{\otimes n}$, with $n\geq 1$,
\[(\pi\otimes \id)\circ \delta(w)=|w'|\otimes w''.\]
Hence, as $V$ is a commutative bialgebra,
\begin{align*}
(\pi\otimes \id)\circ \delta \circ \squplus(x\otimes y)&=|(x\squplus y)'|\otimes (x\squplus y)''\\
&=|x'|\cdot |y'|\otimes (x''\squplus y'')\\
&=(\pi \otimes \id)\circ \squplus_{13,24}\circ \delta(x\otimes y).
\end{align*}
We obtain that 
\[\delta \circ \squplus(x\otimes y)=\squplus_{13,24}\circ (\delta \otimes \delta)(x\otimes y).\]

\textit{Third step}. Let us prove that $(\id \otimes \delta)\circ \delta(x)=(\delta \otimes \id)\circ \delta(x)$
for any $x\in V^{\otimes n}$ by induction on $n$. It is obvious if $n=0$, taking then $x=1$. 
Let us assume the result at all ranks $<n$. The first step implies that
\[(\tdelta \otimes \id)\circ \delta=\squplus_{1,3,24}\circ (\delta \otimes \delta)\circ \tdelta,\]
so
\begin{align*}
(\tdelta \otimes \id\otimes \id)\circ (\delta \otimes \id)\circ \delta(x)
&=\squplus_{1,3,24,5}\circ (\delta \otimes \delta\otimes \id)\circ (\tdelta\otimes \id)\circ \delta(x)\\
&=\squplus_{1,3,24,5}\circ (\delta \otimes \delta\otimes \id)\circ \squplus_{1,3,24}\circ
(\delta\otimes \delta)\circ \tdelta(x)\\
&=\squplus_{1,4,25,36}\circ (\delta \otimes \id \otimes \delta \otimes \id)\circ (\delta\otimes \delta)\circ \tdelta(x),
\end{align*}
whereas
\begin{align*}
(\tdelta \otimes \id\otimes \id)\circ (\id \otimes \delta)\circ \delta(x)
&=(\id \otimes \id\otimes \tdelta)\circ (\tdelta \otimes \id)\circ \delta(x)\\
&=(\id \otimes \id\otimes \tdelta)\circ \squplus_{1,3,24}\circ (\delta \otimes \delta)\circ \tdelta(x)\\
&=\squplus_{1,4,25,36}\circ (\id \otimes \delta\otimes \id \otimes \delta)\circ (\delta \otimes \delta)\circ \tdelta(x).
\end{align*}
By the induction hypothesis, 
\[(\tdelta \otimes \id\otimes \id)\circ (\delta \otimes \id)\circ \delta(x)
=(\tdelta \otimes \id\otimes \id)\circ (\id \otimes \delta)\circ \delta(x),\]
so 
\[(\delta \otimes \id)\circ \delta(x)-(\id \otimes \delta)\circ \delta(x)\in V.\]
Moreover,
\begin{align*}
(\pi\otimes \id)\circ (\delta\otimes \id)\circ \delta(x)&=\left(x^{(1)}\right)'
\otimes \delta\left(\left(x^{(2)}\right)''\right)=(\pi\otimes \id)\circ (\id \otimes \delta)\circ \delta(x).
\end{align*}
Finally, $(\delta\otimes \id)\circ \delta(x)=(\id \otimes \delta)\circ \delta(x)$.\\

\textit{Final step}. It is immediate that $\varepsilon_\Delta$ is a comodule morphism. 
Let us prove now that $\delta$ has a counit. We put, for any $v_1,\ldots, v_n \in V$, with $n\geq 1$,
\[\epsilon_\delta(v_1\ldots v_n)=\begin{cases}
\epsilon_V(v_1)\mbox{ if }n=1,\\
0\mbox{ otherwise.}
\end{cases}\]
Then, if $w=v_1\ldots v_n$,
\begin{align*}
(\epsilon_\delta \otimes \id)\circ \delta(w)&=\epsilon_\delta(|w'|)w''+0\\
&=\epsilon_V(v'_1\cdot \ldots\cdot v'_n)v''_1\ldots v''_n\\
&=\epsilon_V(v'_1)\ldots \epsilon_V(v'_n)v''_1\ldots v''_n\\
&=v_1\ldots v_n,
\end{align*}
whereas
\begin{align*}
(\id \otimes \epsilon_\delta)\circ \delta(w)&=v'_1\ldots v'_n \epsilon_V(v''_1\cdot \ldots \cdot v''_n)+0\\
&=v'_1\ldots v'_n \epsilon_V(v''_1)\ldots \epsilon_V(v''_n)\\
&=v_1\ldots v_n.
\end{align*}
The fact that $\epsilon_V$ is an algebra morphism is left to the reader. \end{proof}

A particular case is obtained when $V$ is the bialgebra of a semigroup $(\Omega,+)$. In this case, 
a basis of the quasishuffle algebra is given by words in $\Omega$. This construction is established in \cite{Ebrahimi-Fard2017-2},
where it is related to Ecalle's mould calculus (product and composition of symmetrel moulds). 
For example, if $k_1,k_2,k_3,k_4\in \Omega$, in this quasishuffle double bialgebra,
\begin{align*}
(k_1)\squplus (k_2k_3k_4)&=(k_1k_2k_3k_4)+(k_2k_1k_3k_4)+(k_2k_3k_1k_4+k_2k_3k_4k_1)\\
&+((k_1+ k_2)k_3k_4)+(k_2(k_1+ k_3)k_4)+k_2k_3(k_1+ k_4)),\\
(k_1k_2)\squplus (k_3k_4)&=(k_1k_2k_3k_4)+(k_1k_3k_2k_4)+(k_1k_3k_4k_2)+(k_3k_1k_2k_4)
+(k_3k_1k_4k_2)+(k_3k_4k_1k_2)\\
&+((k_1+ k_3)k_2k_4)+((k_1+ k_3)k_2k_4)+(k_3(k_1+ k_4)k_2)\\
&+(k_1(k_2+ k_3)k_4)+(k_1k_3(k_2+ k_4))+(k_3k_1(k_2+ k_4))+((k_1+ k_3)(k_2+ k_4)),\\ \\
\Delta((k_1k_2k_3k_4))&=(k_1k_2k_3k_4)\otimes 1+(k_1k_2k_3)\otimes (k_4)+(k_1k_2)\otimes (k_3k_4)\\
&+(k_1)\otimes (k_2k_3k_4)+1\otimes (k_1k_2k_3k_4),\\ \\
\delta((k_1))&=(k_1)\otimes (k_1),\\
\delta((k_1k_2))&=(k_1k_2)\otimes (k_1)\squplus (k_2)+(k_1+k_2)\otimes (k_1k_2),\\
\delta((k_1k_2k_3))&=(k_1k_2k_3)\otimes (k_1)\squplus (k_2)\squplus (k_3)
+((k_1+ k_2)k_3)\otimes (k_1 k_2)\squplus (k_3)\\
&+(k_1(k_2+ k_3))\otimes (k_1)\squplus (k_2 k_3)+(k_1+ k_2+ k_3)\otimes (k_1 k_2 k_3).
\end{align*}
Taking $\Omega=(\N_{>0},+)$, we recover the  double bialgebra of quasisymmetric functions $\QSym$,
A basis of $\QSym$ is given by words in strictly positive integers, which are called compositions. The second coproduct $\delta$ is often called the internal coproduct, and is dual of the Kronecker product of noncommutative symmetric functions.

\subsection{Characters}

\begin{notation}
Let $(B,m,\Delta)$ be a bialgebra. 
\begin{itemize}
\item $B^*$ inherits an algebra structure, with the convolution product $*$ induced by $\Delta$:
\begin{align*}
&\forall \lambda,\mu\in B^*,&\lambda*\mu&=(\lambda\otimes \mu)\circ \Delta.
\end{align*}
The  unit is $\varepsilon_\Delta$. The set of the characters of $B$, that is to say algebra morphisms from $B$ to $\K$,
is denoted by $\chara(B)$. It is a monoid for the convolution product $*$.
\item In the case of a double bialgebra $(B,m,\Delta,\delta)$,  $B^*$ inherits a second convolution product,
denoted by $\star$ and coming from $\delta$:
\begin{align*}
&\forall \lambda,\mu\in B^*,&\lambda\star\mu&=(\lambda\otimes \mu)\circ \delta.
\end{align*}
The  unit is $\epsilon_\delta$.  Moreover, $\chara(B)$ is also a monoid for the convolution product $\star$. 
\item The space of infinitesimal characters of $B$, that is to say $\varepsilon_\Delta$-derivations from $B$ to $\K$,
is denoted by $\infchara(B)$. In other words, a linear map $\lambda:B\longrightarrow \K$ is an infinitesimal character
of $B$ if for any $x,y\in B$,
\[\lambda(xy)=\varepsilon_\Delta(x)\lambda(y)+\lambda(x)\varepsilon_\Delta(y).\]
In other terms, for any $\lambda \in B^*$, $\lambda \in \infchara(B)$ if and only if
$\lambda(\K1_B\oplus B_+^2)=(0)$, where $B_+=\ker(\varepsilon_\Delta)$ is the augmentation ideal of $B$. 
\end{itemize}\end{notation}

If $(B,m,\Delta)$ is a bialgebra, we can consider the transpose $m^*:B^*\longrightarrow (B\otimes B)^*$ of the product $m$.
Note that $B^*\otimes B^*$ is considered as a subspace of $(B\otimes B)^*$, through the canonical injection
\[\left\{\begin{array}{rcl}
B^*\otimes B^*&\longrightarrow&(B\otimes B)^*\\
\lambda \otimes \mu&\longrightarrow&\left\{\begin{array}{rcl}
B\otimes B&\longrightarrow&\K\\
x\otimes y&\longrightarrow&\lambda(x)\mu(y).
\end{array}\right.
\end{array}\right.\]
(This is not an isomorphism except if $B$ is finite-dimensional). As $m$ is a coalgebra morphism, dually 
$m^*:B^*\longrightarrow (B\otimes B)^*$ is an algebra morphism for the convolution products associated
to $\Delta$ on $B$ and $B\otimes B$.  

\begin{prop}
Let $\lambda \in B^*$. Then:
\begin{enumerate}
\item $\lambda \in \chara(B)$ if, and only if, $m^*(\lambda)=\lambda \otimes \lambda$ and $\lambda(1_B)=1$.
\item $\lambda \in \infchara(B)$ if, and only if, $m^*(\lambda)=\lambda \otimes \varepsilon_\Delta+
\varepsilon_\Delta \otimes \lambda$.
\end{enumerate}
\end{prop}

\begin{proof}
Immediate.
\end{proof}

\begin{lemma} \label{lemma1.5}
Let $(B,m,\Delta,\delta)$ be a double bialgebra. For any $\mu\in B^*$,
\[\varepsilon_\Delta\star \mu=\mu(1_B) \varepsilon_\Delta.\]
\end{lemma}

\begin{proof}
As $\varepsilon_\Delta$ is a comodule morphism, 
\[\varepsilon_\Delta \star \mu=(\varepsilon_\Delta \otimes \mu)\circ \delta
=\mu\circ (\varepsilon_\Delta \otimes \id)\circ \delta
=\mu \circ \nu_B \circ \varepsilon_\Delta=\mu(1_B)\varepsilon_\Delta. \qedhere\]
\end{proof}

\begin{prop}
Let $(B,m,\Delta,\delta)$ be a double bialgebra. Then
\begin{align*}
\infchara(B)\star B^*&\subseteq \infchara(B).
\end{align*}
\end{prop}

\begin{proof}
Let $\lambda \in \infchara(B)$ and $\mu\in B^*$. For any $x,y\in B$, using Sweedler's notation $\delta(z)=\sum z'\otimes z''$ for $\delta$,
\begin{align*}
\lambda \star \mu(xy)&=\sum \lambda((xy)')\mu((xy)'')\\
&=\sum\sum \lambda(x'y')\mu(x''y'')\\
&=\sum\sum \lambda(x')\varepsilon_\Delta(y')\mu(x''y'')+\sum\sum\varepsilon_\Delta(x') \lambda(y')\mu(x''y'')\\
&=\sum \lambda(x')\varepsilon_\Delta(y)\mu(x''1_B)+\sum \varepsilon_\Delta(x)\lambda(y')\mu(1_By'')\\
&=\lambda \star \mu(x)\varepsilon_\Delta(y)+\varepsilon_\Delta(x)\lambda\star\mu(y).
\end{align*}
Therefore, $\lambda\star \mu \in \infchara(B)$. 
\end{proof}

\section{General results}

\subsection{Compatibility of the antipode with the coaction}

\begin{prop}\label{prop2.1}
Let $(B,m,\Delta,\delta)$ be a double bialgebra, such that $(B,m,\Delta)$ is a Hopf algebra of antipode $S$.
Then $S$ is a comodule morphism:
\[\delta\circ S=(S\otimes \id)\circ \delta.\]
\end{prop}

\begin{proof}
We consider the space $\homo(B,B\otimes B)$ of linear maps from $B$ to $B\otimes B$,
with the convolution product $*$ defined by 
\[f*g=m_{13,24}\circ (f\otimes g)\circ \Delta.\]
The unit $\iota$ sends any $b\in B$ to $\varepsilon_\Delta(b)1_B\otimes 1_B$.
Let us show that $\delta$ has an inverse in this algebra. 
\begin{align*}
((S\otimes \id)\circ \delta)*\delta&=m_{13,24}\circ (S\otimes \id \otimes \id\otimes \id)\circ (\delta\otimes \delta)
\circ \Delta\\
&=(m\otimes \id)\circ (S\otimes \id\otimes \id)\circ m_{1,3,24}\circ (\delta\otimes \delta)\circ \Delta\\
&=(m\otimes \id)\circ (S\otimes \id\otimes \id)\circ (\Delta \otimes \id)\circ \delta\\
&=((m\circ (S\otimes \id)\circ \Delta)\otimes \id)\circ \delta\\
&=((\nu_B\circ \varepsilon_\Delta)\otimes \id)\circ \delta\\
&=\iota,
\end{align*}
so $(S\otimes \id)\circ \delta$ is a left inverse of $\delta$ for the convolution product $*$.
\begin{align*}
\delta*(\delta \otimes S)&=m_{13,24}\circ (\delta \otimes \delta)\circ (\id \otimes S)\circ \Delta
=\delta \circ m\circ (\id \otimes S)\circ \Delta=\delta \circ \nu_B \circ \varepsilon_\Delta=\iota,
\end{align*}
so $\delta\circ S$ is a right  inverse of $\delta$ for the convolution product $*$.
As $*$ is associative, $\delta$ is invertible and its inverse is $(S\otimes \id)\circ \delta=\delta\circ S$.
\end{proof}

\subsection{From linear forms to endomorphisms}

\begin{notation}
Let $(B,m,\Delta,\delta)$ be a double bialgebra and let $(A,m_A)$ be an algebra. Then the space $\homo(B,A)$
of linear maps from $B$ to $A$ is given two convolution products $*$ and $\star$: for any $f,g\in \homo(B,A)$,
\begin{align*}
f*g&=m_A\circ (f\otimes g)\circ \Delta,&f\star g&=m_A\circ (f\otimes g)\circ \delta.
\end{align*}
The unit of $*$ is $\nu_A\circ \varepsilon_\Delta$ whereas the unit of $\star$ is $\nu_A\circ \epsilon_\delta$.
Two particular examples are given by $A=B$, which defines $*$ and $\star$ for $\en(B)$,
and $A=\K$, giving back the products $*$ and $\star$ on $B^*$. 
\end{notation}

\begin{prop}\label{prop2.2}
Let $(B,m,\Delta,\delta)$ be a double bialgebra. We consider the linear map
\[\Theta:\left\{\begin{array}{rcl}
B^*&\longrightarrow&\en(B)\\
\lambda&\longrightarrow&(\lambda \otimes \id)\circ \delta.
\end{array}\right.\]
For any $\lambda,\mu\in B^*$,
\begin{align*}
\Theta(\lambda*\mu)&=\Theta(\lambda)*\Theta(\mu),&
\Theta(\lambda\star \mu)&=\Theta(\mu)\circ\Theta(\lambda).
\end{align*}
Moreover, $\Theta(\varepsilon_\Delta)=\nu_B\circ \varepsilon_\Delta$ and $\Theta(\epsilon_\delta)=\id_B$. 
The map $\Theta$ is injective, with a left inverse given by
\[\Theta':\left\{\begin{array}{rcl}
\en(B)&\longrightarrow&B^*\\
f&\longrightarrow&\epsilon_\delta\circ f.
\end{array}\right.\]
\end{prop}

\begin{proof}
Let $\lambda,\mu\in B^*$.
\begin{align*}
\Theta(\lambda*\mu)&=(\lambda \otimes \mu\otimes\id)\circ (\Delta\otimes \id)\circ \delta\\
&=(\lambda \otimes \mu\otimes\id)\circ m_{1,3,24}\circ (\delta \otimes \delta)\circ \Delta\\
&=m\circ (\lambda \otimes \id \otimes \mu\otimes \id)\circ (\delta \otimes \delta)\circ \Delta\\
&=m\circ (\Theta(\lambda)\otimes \Theta(\mu))\circ \Delta\\
&=\Theta(\lambda)* \Theta(\mu).\\ \\
\Theta(\lambda\star\mu)&=(\lambda \otimes \mu\otimes\id)\circ (\delta\otimes \id)\circ \delta\\
&=(\lambda \otimes \mu\otimes\id)\circ (\id\otimes \delta)\circ \delta\\
&= (\mu \otimes \id)\circ \delta \circ  (\lambda \otimes \id)\circ \delta\\
&=\Theta(\mu)\circ \Theta(\lambda).
\end{align*}
By definition of the counit, $\Theta(\epsilon_\delta)=\id_B$. As $\varepsilon_\Delta$ is a comodule morphism,
$\Theta(\varepsilon_\Delta)=\nu_B\circ \varepsilon_\Delta$.\\

Let $\lambda \in B^*$.
\[\Theta'\circ \Theta(\lambda)=\epsilon_\delta\circ(\lambda \otimes \id)\circ \delta
=(\lambda \otimes \epsilon_\delta)\circ \delta=\lambda \star \epsilon_\delta=\lambda.\]
So $\Theta'\circ \Theta=\id_{B^*}$. 
\end{proof}

\begin{cor} \label{cor2.3}
Let $(B,m,\Delta,\delta)$ be a double bialgebra. Then $(B,m,\Delta)$ is a Hopf algebra if, and only if,
$\epsilon_\delta$ has an inverse in the algebra $(B^*,*)$. If this holds, the antipode of $(B,m,\Delta)$ is 
\[S=(\epsilon_\delta^{*-1}\otimes \id)\circ \delta.\] 
\end{cor}

\begin{proof}
$\Longrightarrow$. If $(B,m,\Delta)$ is a Hopf algebra, denoting by $S$ its antipode, the inverse of 
$\epsilon_\delta$ in $(B^*,*)$ is $\epsilon_\delta\circ S$. \\

$\Longleftarrow$. If so, putting $S=\Theta(\epsilon_\delta^{*-1})$, we obtain 
\[S*\id_B=\Theta(\epsilon_\delta^{*-1})*\Theta(\epsilon_\delta)=
\Theta(\epsilon_\delta^{*-1}*\epsilon_\delta)
=\Theta(\varepsilon_\Delta)=\nu_B\circ \varepsilon_\Delta.\]
Similarly, $\id_B*S=\nu_B\circ \varepsilon_\Delta$, so $(B,m,\Delta)$ is a Hopf algebra of antipode $S$. 
\end{proof}

\begin{cor}
Let $(B,m,\Delta,\delta)$ be a double bialgebra, such that $(B,m,\Delta)$ is a Hopf algebra.
Then $(B,m)$ is commutative.
\end{cor}

\begin{proof}
As $(B,m,\Delta)$ is a Hopf algebra, $\epsilon_\delta$ has an inverse for the convolution product $*$,
and the antipode of $(B,m,\Delta)$ is $S=(\epsilon_\delta^{*-1}\otimes\id)\circ \delta$ by Corollary \ref{cor2.3}.
As $\epsilon_\delta$ is a character of $(B,m,\Delta)$, its inverse $\epsilon_\delta^{*-1}$ is also a character.
By composition, $S$ is an algebra endomorphism of $B$. By the classical result on the antipode
\cite{Abe1980,Sweedler1969}, it is also a algebra anti-endomorphism. Hence, $S(B)$ is a commutative subalgebra of $B$.
It is enough to prove that $S$ is surjective. By Lemma \ref{lemma1.5},
\begin{align*}
(\epsilon_\delta*\epsilon_\delta^{*-1})\star \epsilon_\delta^{*-1}
&=\varepsilon_\Delta\star \epsilon_\delta^{*-1}
=\epsilon_\delta^{*-1}(1_B)\varepsilon_\Delta=\varepsilon_\Delta,
\end{align*}
and 
\begin{align*}
(\epsilon_\delta*\epsilon_\delta^{*-1})\star \epsilon_\delta^{*-1}&=
(\epsilon_\delta\star \epsilon_\delta^{*-1})*(\epsilon_\delta^{*-1}\star \epsilon_\delta^{*-1})
=\epsilon_\delta^{*-1}*(\epsilon_\delta^{*-1}\star \epsilon_\delta^{*-1}).
\end{align*}
Hence, 
\[\epsilon_\delta^{*-1}*(\epsilon_\delta^{*-1}\star \epsilon_\delta^{*-1})=\varepsilon_\Delta,\]
which implies that $\epsilon_\delta^{*-1}\star \epsilon_\delta^{*-1}=\epsilon_\delta$.
Applying $\Theta$, we obtain that
\[\Theta(\epsilon_\delta^{*-1}\star \epsilon_\delta^{*-1})
=\Theta(\epsilon_\delta^{*-1})\circ \Theta(\epsilon_\delta^{*-1})
=S\circ S=\Theta(\epsilon_\delta)=\id,\]
so $S$ is involutive and therefore, surjective. 
\end{proof}

\subsection{Actions of the groups of characters}

\begin{prop} \label{prop2.5}
Let $(B,m,\Delta,\delta)$ be a double bialgebra and $V$ be a vector space. The following map defines a (right) action
of the monoid $(\chara(B),\star)$ on the space $\homo(B,V)$ of linear maps from $B$ to $V$:
\[\left\{\begin{array}{rcl}
 \homo(B,V)\times \chara(B)&\longrightarrow&\homo(B,V)\\
(f,\lambda)&\longrightarrow&f\leftsquigarrow \lambda=(f\otimes \lambda)\circ \delta.
\end{array}\right.\]
Moreover:
\begin{enumerate}
\item If $A$ is an algebra, $\lambda \in \chara(B)$ and $f:B\longrightarrow A$ is an algebra morphism,
then $f\leftsquigarrow \lambda$ is an algebra morphism.
\item If $C$ is a coalgebra, $\lambda \in \chara(B)$ and $f:B\longrightarrow C$ is a coalgebra morphism,
then $f\leftsquigarrow \lambda$ is a coalgebra morphism.
\item If $B'$ is a bialgebra, $\lambda \in \chara(B)$ and $f:B\longrightarrow B'$ is a bialgebra morphism,
then $f\leftsquigarrow \lambda$ is a bialgebra morphism.
\end{enumerate}
\end{prop}

\begin{proof}
The fact that this is an action comes from the coassociativity of $\delta$. \\

1. By composition, $(\id \otimes \lambda)\circ \delta$ is an algebra morphism.\\

2. We obtain, as $\Delta$ is a comodule morphism,
\begin{align*}
\Delta\circ (f\leftsquigarrow \lambda)&=\Delta \circ (f\otimes \lambda)\circ \delta\\
&=(f\otimes f\otimes \lambda)\circ (\Delta \otimes \id)\circ \delta\\
&=(f\otimes f\otimes \lambda)\circ m_{1,3,24}\circ (\delta\otimes \delta)\circ \Delta \\
&=(f\otimes\lambda\otimes  f\otimes \lambda)\circ (\delta\otimes \delta)\circ \Delta \\
&=((f\leftsquigarrow \lambda)\otimes (f\leftsquigarrow \lambda))\circ \Delta.
\end{align*}
As $\varepsilon_\Delta$ is a comodule morphism,
\begin{align*}
\varepsilon_\Delta \circ (f\leftsquigarrow \lambda)&=((\varepsilon_\Delta \circ f)\otimes \lambda)\circ \delta
=\lambda\circ \nu_B\circ \varepsilon_\Delta=\varepsilon_\Delta.
\end{align*}
Therefore, $f\leftsquigarrow \lambda$ is a coalgebra morphism.\\

3. Direct consequence of 1. and 2. \end{proof}

\begin{remark}
Consequently, if $V$ is an algebra (respectively a bialgebra or a coalgebra), then $\leftsquigarrow$ defines an action
of the monoid $(\chara(A),\star)$ on the set $\homo_a(B,V)$ (respectively  $\homo_b(B,V)$  or $\homo_c(B,V)$)
of morphisms of algebras (respectively bialgebras or coalgebras), from $B$ to $V$. 
\end{remark}

\begin{prop}\label{prop2.6}
Let $(B,m,\Delta,\delta)$ be a double bialgebra, $V$ and $W$ be two spaces and $f:B\longrightarrow V$,
$g:V\longrightarrow W$ be two linear maps. Then
\[(f\circ g)\leftsquigarrow \lambda=f\circ (g\leftsquigarrow \lambda).\]
\end{prop}

\begin{proof}
Indeed,
\[(f\circ g)\leftsquigarrow \lambda=((f\circ g)\otimes \lambda)\circ \delta
=f\circ ((g\otimes \id)\circ \delta)=f\circ (g\leftsquigarrow \lambda). \qedhere\]
\end{proof}

\begin{prop}
Let $(A,m_A)$ be an algebra. For any $f,g\in \homo(B,A)$, for any $\lambda \in\chara(B)$,
\[(f*g)\leftsquigarrow \lambda=(f\leftsquigarrow \lambda)*(g\leftsquigarrow \lambda).\]
\end{prop}

\begin{proof}
Indeed,
\begin{align*}
(f*g)\leftsquigarrow \lambda&=m_A\circ (f\otimes g\otimes \lambda)\circ (\Delta \otimes \id)\circ \delta\\
&=m_A\circ (f\otimes g\otimes \lambda) \otimes m_{1,3,24}\circ (\delta \otimes \delta)\circ \Delta\\
&=m_A\circ (f\otimes \lambda \otimes g\otimes \lambda)\circ (\delta \otimes \delta)\circ \Delta\\
&=m_A((f\leftsquigarrow \lambda)\otimes (g\leftsquigarrow \lambda))\circ \Delta\\
&=(f\leftsquigarrow \lambda)*(g\leftsquigarrow \lambda). \qedhere
\end{align*}\end{proof}

\begin{remark}
In the particular case where $V=\K$, then $\leftsquigarrow=\star$. We obtain that for any $\lambda_1,\lambda_2\in B^*$,
for any $\mu \in \chara(B)$,
\[(\lambda_1*\lambda_2)\star \mu=(\lambda_1\star \mu)*(\lambda_2\star \mu).\]
So $(\chara(B),\star)$ acts on $(B,*)$ by algebra endomorphisms. By restriction, 
$(\chara(B),\star)$ acts on $(\chara(B),*)$ by monoid endomorphisms.
\end{remark}

\section{Connected double bialgebras}

\subsection{Reminders on connected bialgebras}

\begin{notation}
Let $(B,m,\Delta)$ be a bialgebra. We denote by $B_+$ its augmentation ideal, that is to say the kernel of its counit 
$\varepsilon_\Delta$. We define a coassociative (non counitary) coproduct $\tdelta:B_+\longrightarrow B_+\otimes B_+$
by
\begin{align*}
&\forall x\in B_+,&\tdelta(x)=\Delta(x)-x\otimes 1-1\otimes x.
\end{align*}
We may extend $\tdelta$ to $B$ by putting $\tdelta(1_B)=0$. 
The iterated reduced coproducts $\tdelta^{(n)}:B_+\longrightarrow B_+^{\otimes (n+1)}$
are inductively defined by
\begin{align*}
\tdelta^{(n)}&=\begin{cases}
\id_{B^+}\mbox{ if }n=0,\\
\left(\tdelta^{(n-1)}\otimes \id\right)\circ \tdelta\mbox{ if }n\geq 1.
\end{cases}
\end{align*}
In particular, $\tdelta^{(1)}=\tdelta$.
\end{notation}

Recall that a bialgebra $(B,m,\Delta)$ is connected if its coradical is reduced to $\K$. This is equivalent to the fact
that $\tdelta$ is locally nilpotent: for any $x\in B_+$, there exists $n\in \N$ such that 
$\tdelta^{(n)}(x)=0$. In this case, we obtain a filtration of $B$ defined by
\[B_{\leq n}=\ker\left(\tdelta^{(n)}\right)\oplus \K1_B.\]
This is called the \emph{coradical filtration}. 
In particular, $B_{\leq 0}=\K 1_B$ and $B_{\leq 1}\cap B_+=\ker(\tdelta)=\prim(B)$, 
the space of primitive elements of $B$. 
The degree associated to this filtration is denoted by $\deg_p$
\begin{align*}
&\forall x\in B,&\deg_p(x)&=\min(n\in\N,\: x\in B_{\leq n}).
\end{align*}
The coassociativity of $\Delta$ implies that for all $n\in \N$,
\[\Delta(B_{\leq n})\subseteq \sum_{k=0}^n B_{\leq k}\otimes B_{\leq n-k}.\]
Combined with the connectivity of $B$, this gives that for any $n\in \N$,
\[\tdelta(B_{\leq n})\subseteq \sum_{k=1}^{n-1} B_{\leq k}\otimes B_{\leq n-k}.\]
The compatibility of $\Delta$ and $m$ implies that for any $k,l\in \N$,
\[m(B_{\leq k}\otimes B_{\leq l})\subseteq B_{\leq k+l}.\]
Conversely, if $B$ has an increasing filtration $(B_{\leq n})_{n\in \N}$ (which may not be the coradical filtration),
such that for any $k$, $l$, $n\in \N$,
\begin{align*}
m(B_{\leq k}\otimes B_{\leq l})&\subseteq B_{\leq k+l},&
\Delta(B_{\leq n})&\subseteq \sum_{k=0}^n B_{\leq k}\otimes B_{\leq n-k},
\end{align*}
and such that $B_{\leq 0}=\K1_B$, then $B$ is connected, as the coradical of $B$ is necessarily included in $B_{\leq 0}$. 

\begin{example}
In $(\calH_\gr,m,\Delta)$, for any graph $G$,
\[\tdelta^{(k-1)}(G)=\sum_{\substack{I_1\sqcup\ldots \sqcup I_k=V(G),\\ I_1,\ldots,I_k\neq \emptyset}}
G_{\mid I_1}\otimes \ldots \otimes G_{\mid I_k}.\]
In particular, if $k>|V(G)|$, $\tdelta^{(k-1)}(G)=0$, so $\calH_\gr$ is connected.\\
\end{example}

Let $V$ be a vector space. For any map $f:B\longrightarrow V$, we define its valuation by
\[\val(f)=\min\{n\in \N,\: f(B_{\leq n})\neq (0)\},\]
with the convention that $\val(0)=+\infty$. 
Note that for any $f,g\in \homo(B,V)$,
\[\val(f+g)\geq \min(\val(f),\val(g)).\]
We therefore obtain a distance on $\homo(B,V)$ defined by
\[d(f,g)=2^{-\val(f-g)},\]
with the convention $2^{-\infty}=0$. Note that for any $f,g,h\in \homo(B,V)$,
\[d(f,h)\leq \max(d(f,g),d(g,h))\leq d(f,g)+d(g,h).\]

\begin{lemma}
For any vector space $V$, $(\homo(B,V),d)$ is a complete metric space.
\end{lemma}

\begin{proof}
Let $(f_n)_{n\in \N}$ be a Cauchy sequence of $\homo(B,V)$. For any $n\in \N$, there exists $N(n)$ such that 
if $k,l\geq N(n)$, then $(f_k)_{\mid B_{\leq n}}=(g_k)_{\mid B_{\leq n}}$. Let us fix for any $n$,
a complement $B_n$ of $B_{\leqslant n-1}$ in $B_{\leqslant n}$. Then for any $n\in \N$,
\[B_{\leq n}=\bigoplus_{k=0}^n B_k,\]
and consequently 
\[B=\bigoplus_{k=0}^\infty B_k.\]
Let $n\in \N$. We define $g^{(n)}:B_n\longrightarrow V$ by
\[g^{(n)}=(f_{N(n)})_{\mid B_n},\]
 and we consider the map 
 \[g=\bigoplus_{k=0}^\infty g^{(k)}.\]
 If $k\geq \max (N(0),\ldots,N(n))$, then $(f_k)_{\mid B_{\leq n}}=g_{\mid B_{\leq n}}$,
 so $d(f_k,g)\leqslant 2^{-n}$. Hence, $(f_n)_{n\in \N}$ converges to $g$.
\end{proof}

\begin{prop}
Let $(B,m,\Delta)$ be a connected bialgebra and let $(A,m_A)$ be an algebra. 
For any $f,g\in \homo(B,A)$,
\[\val(f*g)\geq \val(f)+\val(g).\]
Consequently, $*:\homo(B\otimes B,A)\longrightarrow \homo(B,A)$ is continuous.
\end{prop}

\begin{proof}
Let $n<\val(f)+\val(g)$. Then
\begin{align*}
f*g(B_{\leq n})&=m_A\circ (f\otimes g)\circ \Delta(B_{\leq n})\\
&\subseteq \sum_{k=0}^n m_A (f(B_{\leq k})\otimes g(B_{\leq n-k}))\\
&=(0),
\end{align*}
as either $k<\val(f)$ or $n-k<\val(g)$. Hence, $\val(f*g)\geq \val(f)+\val(g)$. 
\end{proof}

Consequently, if $A$ is an algebra and $f:B\longrightarrow A$ is a map such that $\val(f)\geq 1$,
for any $n\in \N$,  $\val(f^{*n})\geq n$. Let $(a_n)_{n\in \N}$ be a sequence of scalars. For any $n,p\in \N$,
\[\val\left(\sum_{k=n}^{n+p} a_k f^{*k}\right)
\geq \min (\val(a_k f^{*k}),\: k\in \{n,\ldots,n+p\})\geq n.\]
Hence, as $(\homo(B,A),d)$ is complete, the series $\sum a_k f^{*k}$ converge in $\homo(B,A)$. 
We obtain:

\begin{prop}\label{prop3.3}
Let $(B,m,\Delta)$ be a connected bialgebra and let $A$ be an algebra. For any $f\in \homo(B,A)$
such that $f(1_B)=0$, we obtain a continous algebra morphism
\[\ev_f:\left\{\begin{array}{rcl}
\K[[T]]&\longrightarrow&\homo(B,A)\\
\displaystyle \sum_{k=0}^\infty a_kT^k&\longrightarrow&\displaystyle \sum_{k=0}^\infty a_kf^{*k}.
\end{array}\right.\]
Moreover, for any $\displaystyle P(T)=\sum_{k=0}^\infty a_k T^k\in \K[[T]]$,
$\ev_f(P(T))(1_B)=g(0)1_A$ and for any $x\in B_+$,
\[\ev_f(P(T))(x)=\sum_{k=1}^{\val(x)} a_k m_A^{(k-1)}\circ f^{\otimes k}\circ \tdelta^{(k-1)}(x).\]
\end{prop}

\begin{proof}
As $f(1_B)=0$, $\val(f)\geq 1$: $\ev_f$ is well-defined. For any $k,l\in \N$,
\[\ev_f(T^kT^l)=\ev_f(T^{k+l})=f^{*(k+l)}=f^{*k}*f^{*l}=\ev_f(T^k)*\ev_f(T^l).\]
By linearity, if $P(T),Q(T)\in \K[X]$, $\ev_f(P(T)Q(T))=\ev_f(P(T))*\ev_f(Q(T))$. 
By continuity and density of $\K[T]$ in $\K[[T]]$, this is still true if $P(T)$, $Q(T)\in\K[[T]]$. 
As $f(1_B)=0$, for any $x\in B_+$,
\[f^{*k}(x)=m_A^{(k-1)}\circ f^{\otimes k}\circ \Delta^{(k-1)}(x)
=m_A^{(k-1)}\circ f^{\otimes k}\circ \tdelta^{(k-1)}(x),\]
which implies the announced formula for $\ev_f(P(T))(x)$. 
\end{proof}

\begin{notation} 
We shall write, for any $P(T)\in \K[[T]]$ and $f\in \homo(B,A)$ such that $f(1_B)=0$,
\[P(f)=\ev_f(P(T)).\]
Note that  for any $P(T)$, $Q(T)\in \K[[T]]$, $PQ(f)=P(f)*Q(f)$.\\
\end{notation}

In particular, taking $A=B$ and $\rho$ the canonical projection on $B_+$ which vanishes on $\K1_B$,
we can consider 
\[S=\ev_\rho\left(\frac{1}{1+X}\right)=\frac{1}{1+\rho}=\sum_{k=0}^\infty (-1)^k\rho^{*k}.\]
Then $S$ is the inverse of $\nu_B\circ \varepsilon_\Delta+\rho=\id$ for the convolution product: 
we proved that $(B,m,\Delta)$ is a Hopf algebra and recovered Takeuchi's formula \cite{Takeuchi1971}: for any $x\in B_+$,
\[S(x)=\sum_{k=1}^\infty (-1)^k m^{(k-1)}\circ \tdelta^{(k-1)}(x).\]

\begin{lemma}\label{lemma3.4}
Let $(B,m,\Delta)$ be a connected bialgebra and let $A$ be an algebra. For any $f\in \homo(B,A)$
such that $f(A_B)=0$ and for any formal series $P,Q\in \K[[T]]$, such that $Q(0)=0$,
\[\ev_f(P\circ Q(T))=\ev_{\ev_f(Q(T))}(P(T)).\]
In other words, $(P\circ Q)(f)=P(Q(f))$. 
\end{lemma}

\begin{proof}
As $Q$ has no constant term, if $\val(f)\geq 1$,
then $\val(\ev_f(Q(T)))\geq 1$ and $\ev_{\ev_f(Q(T))}$ exists.\\

We start with the particular case $P=X^n$, for a certain $n\in \N$. As $\ev_f$ is an algebra morphism, 
\begin{align*}
\ev_f(X^n\circ Q)&=\ev_f(Q^n)=\ev_f(Q)^{*n}=\ev_{f(Q)}(X^n).
\end{align*}
By linearity of $\ev_f(P\circ Q)$ and of $\ev_{\ev_f(Q)}(P)$, the equality is still true if $P\in \K[X]$. 
By continuity of $P\longrightarrow\ev_f(P\circ Q)$ and of $P\longrightarrow\ev_{\ev_f(Q)}(P)$, 
as $\K[T]$ is dense in $\K[[T]]$, this remains true for any $P\in \K[[T]]$. 
\end{proof}

\subsection{Applications to shuffle and quasishuffle bialgebras}

\begin{prop}[Universal property of shuffle bialgebras] \label{prop3.5}
Let $(B,m,\Delta)$ be a connected bialgebra,
$V$ be a vector space, $\phi:B\longrightarrow V$ be a linear map such that $\phi(1_B)=0$. 
We consider the shuffle bialgebra $(T(V),\shuffle,\Delta)$ or the quasishuffle bialgebra $(T(V),\squplus,\Delta)$
if $V$ is a (non necessarily unitary) algebra. 
We equip the tensor coalgebra $T(V)$ with the concatenation product, and the associated convolution
on $\hom(B,T(V))$ is denoted by $*$.
Then $\Phi=\dfrac{1}{1-\phi}$ is the unique coalgebra map making the following diagram commuting:
\[\xymatrix{(B,\Delta)\ar[dr]_\phi\ar[r]^\Phi &((T(V),\Delta)\ar[d]^\pi\\
&V}\]
where $\pi$ is the canonical projection onto $V$. Moreover:
\begin{enumerate}
\item $\Phi$ is injective if, and only if, $\phi_{\mid \prim(B)}$ is injective. 
\item $\Phi$ is a bialgebra morphism from $(B,m,\Delta)$ to $(T(V),\shuffle,\Delta)$ if, and only if,
$\phi(B_+^2)=0$, where $B_+$ is the augmentation ideal of $B$.
\item If $(V,\cdot)$ is an algebra (not necessarily unitary),  then $\Phi$ is a bialgebra morphism 
from $(B,m,\Delta)$ to $(T(V),\squplus,\Delta)$ if, and only if, for any $x,y\in B_+$, $\phi(xy)=\phi(x)\cdot \phi(y)$. 
\end{enumerate}
\end{prop}

\begin{proof}
Firstly, observe that as $\phi(1_B)=0$, $\val(\phi)\geq 1$ and $\Phi$ exists. 
Let us prove that $\Phi$ is a coalgebra morphism. Firstly, as $\Phi(1_B)=1$ is a group-like,
\[\Delta \circ \Phi(1_B)=(\Phi\otimes \Phi)\circ \Delta(1_B)=1\otimes 1.\]
Let $x\in B_+$.
\begin{align*}
\tdelta\circ \Phi(x)&=\tdelta\left(\sum_{k=1}^\infty f^{\otimes k}\circ \tdelta^{(k-1)}(x)\right)\\
&=\sum_{k=1}^\infty \sum_{i=1}^{k-1}\left( f^{\otimes i}\otimes f^{\otimes (k-i)}\right)\circ \tdelta^{(k-1)}(x)\\
&=\sum_{k=1}^\infty \sum_{i=1}^{k-1}\left( f^{\otimes i}\otimes f^{\otimes (k-i)}\right)\circ 
\left(\tdelta^{(i-1)}\otimes \tdelta^{(k-i-1)}\right)\circ \tdelta(x)\\
&=\sum_{i,j=1}^\infty\left( f^{\otimes i}\otimes f^{\otimes j}\right)\circ 
\left(\tdelta^{ (i-1)}\otimes \tdelta^{(j-1)}\right)\circ \tdelta(x)\\
&=(\Phi\otimes \Phi)\circ \tdelta(x),
\end{align*}
so $\Phi$ is indeed a coalgebra morphism. Moreover, for any $x\in B_+$,
\[\varpi\circ \Phi(x)=\phi(x)+0=\phi(x).\]
As $\pi\circ\Phi(1_B)=\pi(1)=0=\phi(1_B)$, $\pi\circ \Phi=\phi$.\\

Let $\Psi:(B,\Delta)\longrightarrow (T(V),\Delta)$ be another coalgebra morphism, such that 
$\pi\circ \Psi=\pi\circ \Phi=\phi$. As $1$ is the unique group-like element of $T(V)$, $\Phi(1_B)=\Psi(1_B)=1$.
Let us assume that $\Phi\neq \Psi$. There exists $x\in B_+$, such that $\Phi(x)\neq \Psi(x)$. 
Let us choose such an $x$,  with $\deg_p(x)=n$ minimal. As $\tdelta(x)\in B_{\leq n-1}^{\otimes 2}$,
by definition of $n$,
\[\tdelta \circ \Phi(x)=(\Phi\otimes \Phi)\circ \tdelta(x)=(\Psi\otimes \Psi)\circ \tdelta(x)=
\tdelta \circ \Psi(x),\]
so $\Phi(x)-\Psi(x)\in \ker(\tdelta)=V$. Hence, $\Phi(x)-\Psi(x)=\pi\circ \Phi(x)-\pi\circ \Psi(x)=0$:
this is a contradiction, so $\Phi=\Psi$. \\

1. $\Longrightarrow$. If $\Phi$ is injective, by restriction $\Phi_{\mid \prim(B)}$ is injective. If $x\in \prim(B)$,
then $\Phi(x)\in \prim(T(V))=V$, so $\Phi(x)=\pi\circ \Phi(x)=\phi(x)$: we obtain that $\phi_{\mid \prim(B)}$
is injective.\\

$\Longleftarrow$.  Let us assume that $\Phi$ is not injective. 
Let $x\in \ker(\Phi)\cap B_+$, nonzero,  with $\deg_p(x)=n$ minimal. Then
\[(\Phi\otimes \Phi)\circ \tdelta(x)=\tdelta\circ \Phi(x)=0.\]
By definition of $n$, $\Phi_{\mid B_{\leq n-1}}$ is injective. As $\tdelta(x)\in B_{\leq n-1}\otimes B_{\leq n-1}$,
we obtain that $\tdelta(x)=0$, so $x\in \prim(B)$. Then $\Phi(x)=\phi(x)=0$, so $\phi_{\mid \prim(B)}$
is not injective.\\

3. $\Longrightarrow$. Let us assume that $\Phi:(B,m,\Delta)\longrightarrow (T(V),\squplus,\Delta)$ is a bialgebra
morphism. As $\pi:(T(V)_+,\squplus)\longrightarrow (V,\cdot)$ is an algebra morphism,
by composition $\pi\circ \Phi_{\mid B_+}=\phi_{\mid B_+}$ is an algebra morphism from $(B_+,m)$ to $(V,\cdot)$.\\

$\Longleftarrow$. Let us consider $\Phi_1=\squplus\circ (\Phi\otimes \Phi)$ and $\Phi_2=\Phi\circ m$.
As $m$ and $\squplus$ are coalgebra morphisms, by composition both $\Phi_1$ and $\Phi_2$ are coalgebra
morphisms. In order to prove that $\Phi_1=\Phi_2$, it is enough to prove that 
$\pi\circ \Phi_1=\pi\circ \Phi_2$. Let $x,y\in B_+$. 
\begin{align*}
\pi\circ \Phi_1(1_B\otimes y)&=\pi(1\squplus \Phi(y))=\pi\circ \Phi(y)=\phi(y),\\
\pi\circ \Phi_2(1_B\otimes y)&=\pi\circ \Phi_2(y)=\phi(y),
\end{align*} 
so $\pi\circ \Phi_1(1_B\otimes y)=\pi\circ \Phi_2(1_B\otimes y)$. Similarly, 
$\pi\circ \Phi_1(x\otimes 1_B)=\pi\circ \Phi_2(x\otimes 1_B)$. 
\begin{align*}
\pi\circ \Phi_1(x\otimes y)&=\pi(\Phi(x)\squplus \Phi(y))
=\pi\circ \Phi(x)\cdot \pi\circ \Phi(y)=\phi(x)\cdot \phi(y),\\
\pi\circ \Phi_2(x\otimes y)&=\pi\circ \Phi(xy)=\phi(xy).
\end{align*}
By hypothesis, $\pi\circ \Phi_1(x\otimes y)=\pi\circ \Phi_2(x\otimes y)$,
which gives $\pi\circ \Phi_1=\pi\circ \Phi_2$ and finally $\Phi_1=\Phi_2$: $\Phi$ is an algebra morphism.\\

2. From the second point, with $\cdot=0$. 
\end{proof}

\subsection{Infinitesimal characters and characters}

\begin{prop}\label{prop3.6}
Let $(B,m,\Delta)$ be a connected bialgebra. The following maps are bijections, inverse one from the other:
\begin{align*}
\exp&:\left\{\begin{array}{rcl}
\infchara(B)&\longrightarrow&\chara(B)\\
\lambda&\longrightarrow&\displaystyle e^\lambda=\sum_{k=0}^\infty \frac{1}{k!}\lambda^{*k},
\end{array}\right.\\
\ln&:\left\{\begin{array}{rcl}
\chara(B)&\longrightarrow&\infchara(B)\\
\lambda&\longrightarrow&\displaystyle\ln(1+(\lambda-\varepsilon_\Delta))=\ln(\lambda)
=\sum_{k=0}^\infty \frac{(-1)^{k+1}}{k}(\lambda-\varepsilon_\Delta)^{*k}.
\end{array}\right. \end{align*}\end{prop}

\begin{proof}
We consider the two subsets 
\begin{align*}
B^*_0&=\{\lambda \in B^*\mid \lambda(1_B)=0\},&
B^*_1&=\{\lambda \in B^*\mid \lambda(1_B)=1\},
\end{align*}
and the maps
\begin{align*}
\exp&:\left\{\begin{array}{rcl}
B_0^*&\longrightarrow&B_1^*\\
\lambda&\longrightarrow&e^\lambda=\ev_\lambda(\exp(T)),
\end{array}\right.&
\ln&:\left\{\begin{array}{rcl}
B_1^*&\longrightarrow&B_0^*\\
\lambda&\longrightarrow&\ln(\lambda)=\ev_{\lambda-\varepsilon_\Delta}(\ln(1+T)).
\end{array}\right.
\end{align*}
If $\lambda \in B_0^*$, then $\val(\lambda)\geq 1$, so $\ev_\lambda(\exp(T))$ is well-defined.
Moreover, for any $\lambda\in B_0^*$, $\exp(\lambda)(1_B)=1$, so $\exp$ is well-defined. 
If $\lambda \in B_1^*$, then $(\lambda-\varepsilon_\Delta)(1_B)=0$, 
so $\ev_{\lambda-\varepsilon_\Delta}(\ln(1+T))$ is well-defined. 
Moreover, for any $\lambda\in B_0^*$, $\ln(\lambda)(1_B)=0$, so $\ln$ is well-defined.\\

By Lemma \ref{lemma3.4}, for any $\lambda \in B_0^*$,
\begin{align*}
\ln\circ \exp(\lambda)&=\ev_{\ev_\lambda(\exp(T))-\varepsilon_\Delta}(\ln(1+T))\\
&=\ev_{\ev_\lambda(\exp(T)-1)}(\ln(1+T))\\
&=\ev_\lambda(\ln(1+T)\circ (\exp(T)-1))\\
&=\ev_\lambda(T)\\
&=\lambda.
\end{align*}
Similarly, if $\lambda \in B_1^*$,
\begin{align*}
\exp\circ \ln(\lambda)&=\ev_{\ev_{\lambda-\varepsilon_\Delta}(\ln(1+T))}(\exp(T))\\
&=\ev_{\lambda-\varepsilon_\Delta}(\exp(T)\circ \ln(1+T))\\
&=\ev_{\lambda-\varepsilon_\Delta}(1+T)\\
&=\varepsilon_\Delta+\lambda-\varepsilon_\Delta\\
&=\lambda,
\end{align*}
so $\exp$ and $\ln$ are bijections, inverse one from the other.\\

Let $\lambda \in \infchara(B)$. Then $\lambda(1_B)=0$, so $\infchara(B) \subseteq B_0^*$. 
By definition, $\chara(B)\subseteq B_1^*$.
It remains to prove that for any $\lambda \in B_0^*$, $\exp(\lambda)\in \chara(B)$ if, and only if, 
$\lambda\in \infchara(B)$. We shall use the transpose $m^*$ of the product. As $m$ is a coalgebra morphism,
dually, $m^*$ is an algebra morphism for the product $*$. Let $f\in B^*$, of valuation equal to $N$.
Let $n<N$ and let $x\otimes y\in (B\otimes B)_{\leq n}$. We can assume that $x\in B_{\leq k}$ and $y\in B_{\leq n-k}$,
with $0\leq k\leq n$. 
\[m^*(f)(x\otimes y)=f(xy)\in f(B_{\leq k}B_{\leq n-k})\subseteq f(B_{\leq n})=(0),\]
so $\val(m^*(f))\geq N$: we deduce that $m^*$ is continuous. Hence, for any formal series $P(T)\in \K[[T]]$,
\[m^*(P(\lambda))=m^*(\ev_\lambda(P(T)))=\ev_{m^*(\lambda)}(P(T))=P(m^*(\lambda)).\]
Let us assume that $\lambda\in \infchara(B)$. Then
\begin{align*}
m^*(\exp(\lambda))&=m^*(e^\lambda)\\
&=e^{m^*(\lambda)}\\
&=e^{\lambda \otimes \lambda}\\
&=e^{(\lambda\otimes \varepsilon_\Delta)*(\varepsilon_\Delta\otimes \lambda)}\\
&=e^{\lambda\otimes \varepsilon_\Delta}*e^{\varepsilon_\Delta\otimes \lambda}\\
&=(e^\lambda \otimes \varepsilon_\Delta)*(\varepsilon_\Delta\otimes e^\lambda)\\
&=e^\lambda \otimes e^\lambda\\
&=\exp(\lambda)\otimes \exp(\lambda),
\end{align*}
as $\varepsilon_\Delta \otimes \lambda$ and $\lambda\otimes \varepsilon_\Delta$ commute for the product $*$,
$\varepsilon_\Delta$ being its unit. So $\exp(\lambda)$ is indeed in $\chara(B)$. 

Let us assume that $\exp(\lambda)=\mu\in \chara(B)$. Then
\begin{align*}
m^*(\lambda)&=\ln(1+m^*(\mu-\varepsilon_\Delta))\\
&=\ln(1+\mu\otimes \mu-\varepsilon_\Delta\otimes \varepsilon_\Delta)\\
&=\ln(1+(\mu-\varepsilon_\Delta)\otimes \varepsilon_\Delta
+\varepsilon_\Delta\otimes (\mu-\varepsilon_\Delta)+(\mu-\varepsilon_\Delta)\otimes (\mu-\varepsilon_\Delta))\\
&=\ln(1+(\mu-\varepsilon_\Delta)\otimes \varepsilon_\Delta)
+\ln(1+\varepsilon_\Delta\otimes (\mu-\varepsilon_\Delta))\\
&=\ln(1+\mu-\varepsilon_\Delta)\otimes \varepsilon_\Delta
+\varepsilon_\Delta\otimes \ln(1+\mu-\varepsilon_\Delta)\\
&=\ln(\mu)\otimes \varepsilon_\Delta+\varepsilon_\Delta\otimes \ln(\mu)\\
&=\lambda\otimes \varepsilon_\Delta+\varepsilon_\Delta\otimes \lambda,
\end{align*}
so $\lambda \in \infchara(B)$. 
\end{proof}

\begin{lemma}
\label{lemma3.7}
Let $(B,m,\Delta,\delta)$ be a connected double bialgebra. For any $n\in \N$,
\[\tdelta(B_{\leq n})\subseteq B_{\leq n}\otimes B.\]
\end{lemma}

\begin{proof}
For any $x\in B$, we put $\rho_L(x)=x\otimes 1_B$ and $\rho_R(x)=1_B\otimes x$. 
Then, putting $\delta(x)=x'\otimes x''$,
\begin{align*}
m_{1,3,24}\circ (\delta \otimes \delta)\circ \rho_L(x)&=m_{1,3,24}(x'\otimes x''\otimes 1_B\otimes 1_B)\\
&=x'\otimes 1_B\otimes x''\\
&=(\rho_L\otimes \id)\circ \delta(x),
\end{align*}
so $\rho_L:B\longrightarrow B\otimes B$ is a comodule morphism. Similarly, $\rho_R$ is a comodule morphism.
Hence, $\Delta-\rho_L-\rho_R$ is a comodule morphism. For any $x\in B_+$,
$\tdelta(x)=\Delta(x)-\rho_L(x)-\rho_R(x)$, so $\tdelta:B^+\longrightarrow B^+\otimes B^+$ 
is a comodule morphism. By composition, for any $n\in \N$, $\tdelta^{(n)}$ is a comodule morphism. So
its kernel is a sub-comodule of $B$: for any $n\in \N$,
\[\tdelta\left(\ker\left(\tdelta^{(n)}\right)\right)\subseteq \ker\left(\tdelta^{(n)}\right)\otimes B.\]
The result then follows immediately. 
\end{proof}

\begin{prop}\label{prop3.8}
Let $(B,m,\Delta,\delta)$ be a connected double bialgebra, $A$ an algebra, $f:B\longrightarrow A$ a map
such that $f(1_B)=0$ and $\lambda \in \chara(B)$. For any $P(T)\in \K[[T]]$,
\[P(f)\leftsquigarrow \lambda=P(f\leftsquigarrow \lambda).\]
\end{prop}

\begin{proof}
Firstly, 
\[f\leftsquigarrow \lambda(1_B)=f(1_B)\lambda(1_B)=0,\]
so $\ev_{f\leftsquigarrow \lambda}(P(T))$ is well-defined. Let us first consider the case where $P(T)=T^n$,
with $n\in \N$. Then 
\[\ev_f(T^n)\leftsquigarrow \lambda=(T^{*n})\leftsquigarrow \lambda
=(T\leftsquigarrow \lambda)^{*n}=\ev_{f\leftsquigarrow \lambda}(T^n).\]
By linearity in $P(T)$, for any $P(T)\in \K[T]$, the announced equality is satisfied. \\

Let $V$ be a vector space, $f\in \homo(B,V)$ and let us denote by $N$ its valuation. 
By Lemma \ref{lemma3.7}, if $n<N$,
\[f\leftsquigarrow \lambda(B_{\leq n})=(f\otimes \lambda)\circ \delta(B_{\leq n})
\subseteq f(B_{\leq n})\otimes \lambda(B)=(0),\]
so $\val(f\leftsquigarrow \lambda)\leq \val(f)$. In other words, the following map is continuous:
\[\left\{\begin{array}{rcl}
\homo(B,V)&\longrightarrow&\homo(B,V)\\
f&\longrightarrow&f\leftsquigarrow \lambda.
\end{array}\right.\]
Therefore, by density of $\K[T]$ in $\K[[T]]$, the announced equality is true for any $P\in \K[[T]]$. \end{proof}

\subsection{Polynomial invariants}

\begin{theo}\label{theo3.9}
Let $(B,m,\Delta)$ be a connected bialgebra and let $\lambda\in B^*$, such that $\lambda(1_B)=1$.
\begin{enumerate}
\item 
There exists a unique coalgebra morphism $\Phi_\lambda:(B,m,\Delta)\longrightarrow(\K[X],m,\Delta)$
such that $\epsilon_\delta\circ \Phi_\lambda=\lambda$. Moreover, $\Phi_\lambda=\lambda^X$
is given by $\Phi_\lambda(1_B)=1$ and
\begin{align*}
&\forall x\in B_+,&\Phi_\lambda(x)
&=\lambda^X(x)=\sum_{k=1}^\infty \lambda^{\otimes k}\circ \tdelta^{(k-1)}(x) H_k(X),
\end{align*}
where for any $k\in \N$, $H_k(X)$ is the $k$-th Hilbert polynomial
\[H_k(X)=\frac{X(X-1)\ldots (X-k+1)}{k!}.\]
\item $\Phi_\lambda$ is a bialgebra morphism from $(B,m,\Delta)$ to $(\K[X],m,\Delta)$ if and only if
$\lambda\in \chara(B)$.
\item $\Phi_\lambda$ is a double bialgebra morphism from $(B,m,\Delta,\delta)$ to $(\K[X],m,\Delta,\delta)$ 
if and only if $\lambda=\epsilon_\delta$.
\end{enumerate}
\end{theo}

\begin{proof}
1. \textit{Existence}. 
We extend the scalars to the field $\K((X))$ of fractions of $\K[[X]]$. Then $\K((X))\otimes B$ is a double bialgebra
over $\K((X))$. The map $\lambda$ is extended as a $\K((X))$-linear map from $\K((X))\otimes B$ to $\K((X))$,
which we denote by $\overline{\lambda}$. 
As $\lambda(1_B)-\varepsilon_\Delta(1_B)=1-1=0$, we can consider
\[\overline{\lambda}^X=\ev_{\overline{\lambda}-\overline{\varepsilon_\Delta}}((1+T)^X)
=\sum_{k=0}^\infty(\overline{\lambda}-\overline{\varepsilon_\Delta})^{\otimes k}\circ \Delta^{(k-1)}(x)H_k(X).\]
Therefore, for any $x\in B_+=\K\otimes B_+\subseteq \K((X))\otimes B_+$, 
\[\overline{\lambda}^X(x)=\sum_{k=1}^\infty\overline{\lambda}^{\otimes k}\circ \tdelta^{(k-1)}(x) H_k(X)
=\sum_{k=1}^\infty\lambda^{\otimes k}\circ \tdelta^{(k-1)}(x) H_k(X)\in \K[X].\]
Hence, $\overline{\lambda}^X_{\mid B}=\lambda^X$ takes its values in $\K[X]$.\\

Identifying $\K[X]\otimes \K[X]$ and $\K[X,Y]$, as $(1+T)^{X+Y}=(1+T)^X(1+T)^Y$, 
\begin{align*}
\Delta \circ \lambda^X&=\lambda^{X+Y}=\lambda^X*\lambda^Y=\lambda^X*\lambda^Y
=(\lambda^X\otimes \lambda^X)\circ \Delta,
\end{align*}
so $\lambda^X$ is a coalgebra morphism. Moreover, 
$\epsilon_\delta \circ \lambda^X=(\lambda^X)_{\mid X=1}=\lambda$.\\

\textit{Unicity}. Let $\Lambda:B\longrightarrow \K$ be a coalgebra morphism. We put 
$\epsilon_\delta \circ \Lambda=\lambda$.  
We consider $\Lambda$ as an element of $B^*[[X]]$, putting
\[\Lambda(X)=\sum_{n=0}^\infty f_n X^n,\]
where for any $n\geq 0$, for any $x\in B$, $f_n(x)$ is the coefficient of $X^n$ in $\Lambda(x)$. 
As $\Delta \circ \Lambda=(\Lambda \otimes \Lambda)$, still identifying $\K[X]\otimes \K[X]$ and $\K[X,Y]$,
in $B^*[[X,Y]]$,
\begin{align*}
\Lambda(X+Y)&=\sum_{n=0}^\infty f_n (X+Y)^n=\Delta\circ \Lambda(X)
=(\Lambda(X)\otimes \Lambda(X))\circ \Delta=\Lambda(X)*\Lambda(Y).
\end{align*}
Derivating according to $Y$ and taking $Y=0$, we obtain
\[\Lambda(X)=\Lambda'(0)*\Lambda(X).\]
So $\Lambda(X)=C*e^{\Lambda'(0)X}$, for a certain constant $C\in B^*$. 
As $\Lambda(0)=\varepsilon_\Delta \circ \Lambda
=\varepsilon_\Delta$, $\Lambda(0)=C=\varepsilon_\Delta$, so $\displaystyle\Lambda(X)=e^{\Lambda'(0)X}$. 
We put $\mu=e^{\Lambda'(0)}\in B^*$, then $\Lambda(X)=\mu^X$. Moreover, 
\[\mu=\epsilon_\delta\circ\Lambda(X)=\lambda,\]
so finally $\Lambda=\lambda^X$. \\

2. $\Longrightarrow$. By composition, if $\lambda^X$ is an algebra morphism, then 
$\epsilon_\delta\circ \lambda^X=\lambda$ is an algebra morphism, so $\lambda$ is a character.\\

$\Longleftarrow$. Let us assume that $\lambda$ is a character.  
We put $\mu=\ln(\lambda)$. Then $\mu$ is an infinitesimal character, so $X\mu$ is also an infinitesimal character
of $\K((X))\otimes B$.  As $\lambda^X=\exp(X\mu)$, $\lambda^X$ is a character of $\K((X))\otimes B$,
so its restriction to $B$ is an algebra mophism from $B$ to $\K[X]$. \\

3. $\Longrightarrow$. If $\lambda^X$ is a double bialgebra morphism, then 
$\lambda=\epsilon_\delta\circ \lambda^X=\epsilon_\delta$.\\

$\Longleftarrow$. By the second point, as $\epsilon_\delta$ is a character, $\epsilon_\delta^X$
is a bialgebra morphism from $(B,m,\Delta)$ to $(\K[X],m,\Delta)$. 
We still identify $\K[X]\otimes \K[X]$ and $\K[X,Y]$. For any $\lambda\in \chara(B)$,
by Proposition \ref{prop3.8}, as $\leftsquigarrow=\star$ for $B^*$,
\[\left(\lambda^X\otimes \lambda^X\right)\circ \delta=\lambda^X\star \lambda^Y=(\lambda\star \lambda^Y)^X.\]
In the particular case $\lambda=\epsilon_\delta$, unit of the product $\star$,
\[(\epsilon_\delta^X\otimes \epsilon_\delta^X)\circ \delta=(\epsilon_\delta^X)^Y=\epsilon_\delta^{XY}
=\delta \circ \epsilon_\delta^X.\]
So $\epsilon_\delta^X$ is a double bialgebra morphism. 
\end{proof}

Using the $\exp$ and $\ln$ bijections, we obtain:

\begin{prop}\label{prop3.10}
Let $(B,m,\Delta)$ be a connected bialgebra and let $\mu\in B^*$, such that $\mu(1_B)=0$.
\begin{enumerate}
\item 
There exists a unique coalgebra morphism $\Psi_\mu:(B,m,\Delta)\longrightarrow(\K[X],m,\Delta)$
such that for any $x\in B$, $\Psi_\mu(x)'(0)=\mu(x)$. Moreover, $\Psi_\mu=e^{\mu X}$
is given on any $x\in B_+$ by
\begin{align}\label{eq1}
\Psi_\mu(x)&=e^{\mu X}(x)=\sum_{k=1}^\infty \mu^{\otimes k}\circ \tdelta^{(k-1)}(x) \dfrac{X^k}{k!}.
\end{align}
\item $\Psi_\mu$ is a bialgebra morphism from $(B,m,\Delta)$ to $(\K[X],m,\Delta)$ if and only if
$\mu\in \infchara(B)$.
\item $\Psi_\mu$ is a double bialgebra morphism from $(B,m,\Delta,\delta)$ to $(\K[X],m,\Delta,\delta)$ if and only if
$\mu=\ln(\epsilon_\delta)$.
\end{enumerate}
\end{prop}

\begin{proof}
All can be proved by taking $\lambda=\exp(\mu)$ and $\Psi_\mu=\Phi_{\exp(\mu)}$. 
Let us now prove (\ref{eq1}). 
\[\Psi_\mu=\ev_{\exp(\mu)-\varepsilon_\Delta}((1+T)^X)
=\ev_\mu((e^T)^X)=\ev_\mu(e^{TX})=e^{\mu X}.\]
Therefore, for any $x\in B_+$, as $\mu(1_B)=0$,
\[\Psi_\mu(x)=\sum_{k=0}^\infty \mu^{*k}\dfrac{X^k}{k!}
=\sum_{k=1}^\infty \mu^{\otimes k}\circ \tdelta^{(k-1)} \dfrac{X^k}{k!}.\]
Moreover, 
\[\Psi_\mu(x)'(0)=\mu^{\otimes 1}(x)+0=\mu(x),\]
so $\Psi_\mu(x)'(0)=\mu(x)$. 
\end{proof}

\begin{cor}
Let $(B,m,\Delta,\delta)$ be a connected double bialgebra. There exists a unique double bialgebra morphism
$\Phi$ from $(B,m,\Delta,\delta)$ to $(\K[X],m,\Delta,\delta)$. For any $x\in B_+$,
\[\Phi(x)=\sum_{k=1}^\infty \epsilon_\delta^{\otimes k}\circ \tdelta^{(k-1)}(x) H_k(X).\]
Moreover, for any $\lambda \in \chara(B)$, the unique bialgebra morphism $\Phi_\lambda$ from $(B,m,\Delta)$
to $(\K[X],m,\Delta)$ such that $\epsilon_\delta\circ \Phi_\lambda=\lambda$ is 
\[\Phi_\lambda=\Phi\leftsquigarrow \lambda=(\Phi\otimes \lambda)\circ \delta.\]
\end{cor}

\begin{proof} The first point is a direct reformulation of Theorem \ref{theo3.9}. 
By Proposition \ref{prop2.5}, $\Phi\leftsquigarrow \lambda$ is a bialgebra morphism.
Moreover, by Proposition \ref{prop2.6},
\[\epsilon_\delta \circ (\Phi\leftsquigarrow \lambda)
=(\epsilon_\delta\circ\Phi)\star \lambda=\epsilon_\delta\star \lambda=\lambda.\]
So $\Phi\leftsquigarrow \lambda=\Phi_\lambda$. 
\end{proof}

\begin{cor}\label{theo3.12}
Let $(B,m,\Delta,\delta)$ be a connected double bialgebra and let $\Phi:B\longrightarrow \K[X]$
be the unique double bialgebra morphism. We denote by $\homo_b(B,\K[X])$ the set of bialgebra morphisms
from $(B,m,\Delta)$ to $(\K[X],m,\Delta)$. The following maps are bijections, inverse one from the other:
\begin{align*}
&\left\{\begin{array}{rcl}
\chara(B)&\longrightarrow&\homo_b(B,\K[X])\\
\lambda&\longrightarrow&\Phi\leftsquigarrow \lambda,
\end{array}\right.&
&\left\{\begin{array}{rcl}
\homo_b(B,\K[X])&\longrightarrow&\chara(B)\\
\Psi&\longrightarrow&\epsilon_\delta \circ \Psi.
\end{array}\right.&
\end{align*}\end{cor}

\begin{proof}
Immediate.
\end{proof}

\begin{example}
Let us consider the case of $\calH_\gr$. For any non empty graph $G$,
\[\tdelta^{(k-1)}(G)=\sum_{f:V(G)\twoheadrightarrow[k]} G_{\mid f^{-1}(1)}\otimes \ldots \otimes G_{\mid f^{-1}(k)},\]
therefore
\[\Phi(G)=\sum_{k=1}^\infty \sum_{f:V(G)\twoheadrightarrow[k]} 
\epsilon_\delta(G_{\mid f^{-1}(1)})\ldots \epsilon_\delta(G_{\mid f^{-1}(k)})H_k(X).\]
Moreover, by definition of $\epsilon_\delta$, 
$\epsilon_\delta(G_{\mid f^{-1}(1)})\ldots \epsilon_\delta(G_{\mid f^{-1}(k)})\neq 0$ if, and only if,
for any $i$, $G_{\mid f^{-1}(i)}$ has no edge. This gives us the well-known concept of a valid coloration:
a $k$-coloration is a map $f:V(G)\longrightarrow [k]$; it is packed if $f$ is surjective and it is valid
if for any $\{x,y\}\in E(G)$, $f(x)\neq f(y)$. Hence, denoting by $\PVC(G)$ the set of packed valid coloration of $G$,
\[\Phi(G)=\sum_{f\in \PVC(G)} H_{\max(f)}.\]
Consequently, for any $k\in \N$, $\Phi(G)(n)$ is the number of valid $n$-colorations of $G$:
in other words, $\Phi(G)$ is the chromatic polynomial of $G$ \cite{Harary1969}.
\end{example}

\begin{theo}\label{theo3.13}
The unique double bialgebra morphism $\Phi_{chr}$ from $(\calH_\gr,m,\Delta,\delta)$ to $(\K[X],m,\Delta,\delta)$
sends any graph $G$ to its chromatic polynomial.
\end{theo}

\begin{example}\begin{align*}
\Phi_{chr}(\grun)&=X,&\Phi_{chr}(\grdeux)&=X(X-1),\\
\Phi_{chr}(\grtroisun)&=X(X-1)(X-2),&\Phi_{chr}(\grtroisdeux)&=X(X-1)^2,\\
\Phi_{chr}(\grquatreun)&=X(X-1)(X-2)(X-3),&\Phi_{chr}(\grquatredeux)&=X(X-1)(X-2)^2,\\
\Phi_{chr}(\grquatretrois)&=X(X-1)^2(X-2),&\Phi_{chr}(\grquatrequatre)&=X(X-1)(X^2-3X+3),\\
\Phi_{chr}(\grquatrecinq)&=X(X-1)^3,&\Phi_{chr}(\grquatresix)&=X(X-1)^3.
\end{align*}\end{example}

Let us now consider the case of quasishuffle algebras. Let $(V,\cdot,\delta_V)$ be a commutative (not necessarily
unitary) bialgebra. We denote by $\Phi$ the unique double bialgebra morphism from $(T(V),\squplus,\Delta,\delta)$
to $(\K[X],\squplus,\Delta,\delta)$. For any $v_1,\ldots,v_n \in V$, with $n\geq 1$,
\begin{align*}
\Phi(v_1\ldots v_n)&=\sum_{\substack{v_1\ldots v_n=w_1\ldots w_k,\\ w_1,\ldots,w_k\neq \emptyset}}
\epsilon_\delta(w_1)\ldots \epsilon_\delta(w_k)H_k(X)=\epsilon_V(v_1)\ldots \epsilon_V(v_n)H_n(X)+0.
\end{align*}
Therefore:

\begin{prop} \label{prop3.14}
Let $(V,\cdot,\delta_V)$ be a commutative (not necessarily unitary) bialgebra.
The unique  double bialgebra morphism $\Phi$ from $(T(V),\squplus,\Delta,\delta)$
to $(\K[X],\squplus,\Delta,\delta)$ sends any word $v_1\ldots v_n\in T(V)$ of length $n\geq 1$ to
\[\Phi(v_1\ldots v_n)=\epsilon_V(v_1)\ldots \epsilon_V(v_n)H_n(X).\]
\end{prop}

\begin{remark}
\label{rem3.1}
In the particular case of $\QSym$, the unique double bialgebra morphism from $(\QSym,\squplus,\Delta,\delta)$
to $(\K[X],m,\Delta,\delta)$ sends the composition $(k_1\ldots k_n)$ to $H_n(X)$ for any $n$. 
This morphism is denoted by $\Phi_\QSym$.
\end{remark}

\section{The eulerian idempotent}

\begin{notation}
Let $(B,m,\Delta)$ be a connected bialgebra. Its eulerian idempotent is
\[\varpi=\ev_\rho(\ln(1+T))=\ln(1+\rho)=\sum_{k=1}^\infty \dfrac{(-1)^{k+1}}{k} \rho^{*k}.\]
\end{notation}

\subsection{Logarithm of the counit and the eulerian idempotent}

Let us go back to the map $\Theta$ of Proposition \ref{prop2.2}, with $V=B$. By Lemma \ref{lemma3.7},
it is a continuous algebra map from $B^*$ to $\en(B)$, as it sends $B^*_{\leq n}$ to $\en(B)_{\leq n}$
for any $n$.  

\begin{prop} \label{prop4.1}
Let $(B,m,\Delta,\delta)$ be a connected double bialgebra. Let us denote by $\Phi$ the unique double
bialgebra morphism from $B$ to $\K[X]$. We define an infinitesimal character $\phi\in B^*$ by
\begin{align*}
&\forall x\in B,&\phi(x)&=\Phi(x)'(0),
\end{align*}
 that is to say $\phi(x)$ is the coefficient of $X$ in $\Phi(x)$. Then $\phi=\ln(\epsilon_\delta)$ and 
the eulerian idempotent $\varpi$ of $B$ is
\[\varpi=\Theta(\phi)=(\phi \otimes \id)\circ \delta.\]
\end{prop}

\begin{proof} By the proof of Proposition \ref{prop3.10}, for any $\lambda \in \chara(B)$,
$\Psi_{\ln(\lambda)}=\Phi_{\lambda}$, and for any $x\in B$,
\[\Psi_{\ln(\lambda)}(x)'(0)=\Phi_{\lambda}(x)'(0)=\ln(\lambda)(x).\]
In the particular case where $\lambda=\epsilon_\delta$, then $\Phi=\Phi_{\epsilon_\delta}$ and we obtain that
$\phi=\ln(\epsilon_\delta)$.
As $\Theta$ is a continuous algebra morphism, 
\[\Theta(\phi)=\Theta(\ln(\epsilon_\delta))=\ln(\id)=\varpi. \qedhere\]
\end{proof}

\begin{example} In the case of $\calH_\gr$, this character is denoted by $\phi_{chr}$. 
\begin{align*}
\phi_{chr}(\grun)&=1,&\phi_{chr}(\grdeux)&=-1,&
\phi_{chr}(\grtroisun)&=2,&\phi_{chr}(\grtroisdeux)&=1,\\
\phi_{chr}(\grquatreun)&=-6,&\phi_{chr}(\grquatredeux)&=-4,&
\phi_{chr}(\grquatretrois)&=-2,&\phi_{chr}(\grquatrequatre)&=-3,\\
\phi_{chr}(\grquatrecinq)&=-1,&\phi_{chr}(\grquatresix)&=-1.
\end{align*}\end{example}

\begin{prop}\label{prop4.2}
Let $(B,m,\Delta,\delta)$ be a connected double bialgebra and let $\lambda \in \chara(B)$. Then
\[\ln(\lambda)=\phi\star \lambda.\]
\end{prop}

\begin{proof}
By Proposition \ref{prop3.8} with $V=\K$ (and then $\leftsquigarrow=\star$),
\[\phi\star\lambda=\ln(\epsilon_\delta)\star \lambda=\ln(\epsilon_\delta\star \lambda)=\ln(\lambda). \qedhere\]
\end{proof}

\begin{lemma}\label{lemma4.3}
Let $\mu\in \infchara(B)$. Then $\phi\star\mu=\mu$.
\end{lemma}

\begin{proof}
Let $\lambda_1,\lambda_2 \in B^*$ and $\mu\in \infchara(B)$.
\begin{align*}
(\lambda_1*\lambda_2)\star \mu&=(\lambda_1\otimes \lambda_2\otimes \mu)\circ (\Delta \otimes \id)\circ \delta\\
&=(\lambda_1\otimes \lambda_2\otimes \mu)\circ m_{1,3,24}\circ (\delta\otimes \delta)\circ \Delta\\
&=(\lambda_1\otimes \varepsilon_\Delta \otimes \lambda_2\otimes \mu
+\lambda_1\otimes \mu \otimes \lambda_2\otimes \varepsilon_\Delta)\circ (\delta\otimes \delta)\circ \Delta\\
&=(\lambda_1 \star \varepsilon_\Delta)*(\lambda_2\star\mu)+
(\lambda_1 \star \mu)*(\lambda_2\star\varepsilon_\Delta).
\end{align*}
Hence, for any $n\geq 1$, if $\lambda\in B^*$,
\[\lambda^{*n}\star \mu=\sum_{k=1}^n
(\lambda \star \varepsilon_\Delta)^{*(k-1)}* (\lambda \star \mu)*
(\lambda \star \varepsilon_\Delta)^{*(n-k)}.\]
For $\lambda=\epsilon_\delta-\varepsilon_\Delta$,
\begin{align*}
(\epsilon_\delta-\varepsilon_\Delta)\star \mu&=\epsilon_\delta\star \mu-\varepsilon_\Delta\star 
=\mu-\mu(1_B)\varepsilon_\Delta=\mu,\end{align*}
whereas
\begin{align*}
(\epsilon_\delta-\varepsilon_\Delta)\star \varepsilon_\Delta&=\epsilon_\delta\star \varepsilon_\Delta
-\varepsilon_\Delta\star \varepsilon_\Delta= \varepsilon_\Delta-\varepsilon_\Delta(1) \varepsilon_\Delta=0.
\end{align*}
Therefore, for any $n\geqslant 1$,
\[(\epsilon_\delta-\varepsilon_\Delta)^{*n}\star \mu=\begin{cases}
\mu\mbox{ if }n=1,\\
0\mbox{ otherwise}.
\end{cases}\]
We finally obtain that
\[\phi\star\mu=
\ln(1+\varepsilon_\Delta-\epsilon_\delta)\star\mu=
\sum_{k=1}^\infty \frac{(-1)^{k+1}}{k} (\varepsilon_\Delta-\epsilon_\delta)^*\star \mu=\mu. \qedhere\]
\end{proof}

\begin{prop} \label{prop4.4}
Let $(B,m,\Delta,\delta)$ be a connected double bialgebra. 
Then $\varpi$ is a projector, which kernel is $B_+^2\oplus \K1_B$. 
\end{prop}

\begin{proof}
Indeed, by Lemma \ref{lemma4.3} with $\mu=\phi$,
\[\varpi\circ \varpi=\Theta(\mu)\circ \Theta(\mu)=\Theta(\mu\star \mu)=\Theta(\mu)=\varpi.\]
So $\varpi$ is a projector. As $\phi$ is an infinitesimal character, $\phi(B_+^2\oplus \K1_B)=(0)$.
Moreover, as $\varepsilon_\Delta$ is a comodule morphism, $\delta(B_+)\subseteq B_+\otimes B$,
which implies that
\[\delta(B_+^2\oplus \K1_B)\subseteq (B_+^2\oplus \K1_B)\otimes B.\]
Therefore, as $\varpi=(\phi\otimes \id)\circ \delta$, $B_+^2\oplus \K1_B\subseteq \ker(\varpi)$. \\

Let $x\in B_+$. Then 
\[\varpi(x)=\sum_{k=1}^\infty \dfrac{(-1)^{k+1}}{k} \rho^{*k}(x)
=x+\underbrace{\sum_{k=2}^\infty \dfrac{(-1)^{k+1}}{k} m_{12\ldots k}\circ \tdelta^{(k-1)}(x)}_{\in B_+^2},\]
so $x-\varpi(x)\in B_+^2\oplus \K 1_B$. In particular, if $x\in \ker(\varpi)$, then $x\in B_+^2\oplus \K 1_B$.
Hence, $\ker(\varpi)=B_+^2\oplus \K1_B$. \end{proof}

If $x\in \prim(B)$, then $\varpi(x)=x$, so $\prim(B)\subseteq \im(\varpi)$. 
In the cocommutative case, it is an equality:

\begin{cor}
Let $(B,m,\Delta,\delta)$ be a connected double bialgebra, such that $\Delta$ is cocommutative.
Then the eulerian idempotent $\varpi$ is the projector on $\prim(B)$ which vanishes on $B_+^2\oplus \K 1_B$. 
\end{cor}

\begin{proof}
As $(B,m,\Delta)$ is a cocommutative bialgebra, it is primitively generated by Cartier-Quillen-Milnor-Moore's theorem.
Hence, 
\[B=\prim(B)\oplus B_+^2\oplus \K1_B.\]
As $\prim(B)\subseteq \im(\varpi)$ and $\varpi$ vanishes on $B_+^2\oplus \K1_B$, $\prim(B)=\im(\varpi)$. 
\end{proof}

\begin{example}
Let $G$ be a connected graph. For any $\sim\in \eq(G)$, $G/\sim$ is connected. Hence, as $\calH_\gr$ is cocommutative,
\[\varpi(G)=\sum_{\sim\in \eq_c(G)} \phi_{chr}(G/\sim) G\mid\sim\in \prim(\calH_\gr).\]
If $G$ is not connected, then $\varpi(G)=0$. 
\end{example}

\begin{remark}
If $(B,m,\Delta)$ is neither a commutative or a cocommutative bialgebra, then $\varpi$ is generally not a projector,
and does not vanishes $B_+^2$. To illustrate this, let us consider the  bialgebra freely generated by three generators 
$x_1$, $x_2$, and $y$, with the coproduct defined by
\begin{align*}
\Delta(x_1)&=x_1\otimes 1+1\otimes x_1,\\
\Delta(x_2)&=x_2\otimes 1+1\otimes x_2,\\
\Delta(y)=&y\otimes 1+1\otimes y+x_1\otimes x_2. 
\end{align*}
Then $\varpi(x_1x_2)=\dfrac{[x_1,x_2]}{2}\neq 0$ and $\varpi(y)=y -\dfrac{x_1x_2}{2}$. Therefore, 
\[\varpi^2(y)=y-\dfrac{x_1x_2}{2}-\dfrac{[x_1,x_2]}{4}=y-\dfrac{3x_1x_2-x_2x_1}{4}\neq \varpi(y).\]
\end{remark}

\subsection{Chromatic infinitesimal character}

In the case of graphs, for any infinitesimal character $\mu$, if $\Psi_\mu$ is the associated bialgebra morphism,
for any graph $G$,
\begin{align*}
\Psi_\mu(G)&=\sum_{k=1}^\infty \sum_{V(G)=I_1\sqcup \ldots \sqcup I_k} \mu(G_{\mid I_1})\ldots
\mu(G_{\mid I_k}) \frac{X^k}{k!}.
\end{align*}
As $\mu$ is an infinitesimal character, it vanishes on nonconnected graphs, so this is in fact a sum over 
$\sim \in \eq_c(G)$:
\begin{align*}
\Psi_\mu(G)&=\sum_{\sim\in \eq_c(G)} \prod_{C\in V(G)/\sim} \mu(G_{\mid C}) X^{cl(\sim)},
\end{align*}
where $\cl(\sim)$ is the number of classes of $\sim$. Denoting by $\overline{\mu}$ the character defined
\begin{align*}
&\forall G\in \gr,& \overline{\mu}(G)&=\prod_{\mbox{\scriptsize $H$ connected component of $G$}}
\mu(H),
\end{align*}
 we obtain
\begin{equation}
\label{eq2}
\Psi_\mu(G)=\sum_{\sim\in \eq_c(G)} \overline{\mu}(G\mid \sim) X^{cl(\sim)},
\end{equation}
Let $\phi_{chr}$ be the infinitesimal  character associated to the morphism $\Phi_{chr}$ from $\calH_\gr$
to $\K[X]$: for any graph $G$, 
\[\phi_{chr}(G)=\Phi_{chr}(G)'(0)=\ln(\epsilon_\delta)(G).\]
We obtain from (\ref{eq2}) that for any graph $G$,
\[\Phi_{chr}(G)=\sum_{\sim\in \eq_c(G)} \overline{\phi_{chr}}(G\mid \sim)X^{\cl(\sim)}.\]

\begin{notation}
We shall use here the notion of acyclic orientation of $G$. Recall that:
\begin{itemize}
\item An oriented graph is a pair $G=(V(G),A(G))$, where $V(G)$ is a finite set called the set of vertices of $G$
and $A(G)$ is a set of couples of distinct elements of $G$, such that for any $x,y\in V(G)$, distinct,
\[(x,y)\in A(G)\Longrightarrow (y,x)\notin A(G).\]
A walk in $G$ is a sequence of vertices $(x_1,\ldots,x_k)$ such that for any $i\in [k-1]$,
$(x_i,x_{i+1})\in A(G)$. A cycle $(x_1,\ldots,x_k)$ is a walk if $k\geq 2$ and if $x_1=x_k$. 
The oriented graph is acyclic if it does not contain any cycle. 
\item If $G$ is an oriented graph, its support
is the graph $\supp(G)$ defined by
\begin{align*}
V(\supp(G))&=V(G),&
E(\supp(G))&=\{\{x,y\}\mid (x,y)\in A(G)\}.
\end{align*}
\item If $G$ is an oriented graph, a source of $G$ is a vertex $y\in V(G)$ such that for any $x\in V(G)$,
$(x,y)\notin A(G)$. The set of sources of $G$ is denoted by $s(G)$. 
It is not difficult to show that any non empty acyclic oriented graph has at least one source. 
Consequently, if $G$ is a non empty acyclic oriented graph, then any of its connected component  is also an acyclic oriented
graph and so contains at least one source. Therefore, if $G$ is not connected, $|s(G)|\neq 1$. 
\item If $G$ is a graph, we denote by $O(G)$ the set of acyclic oriented graphs $H$ such that $\supp(H)=G$.
If $x\in V(G)$, we denote by $O(G,x)$ the set of acyclic oriented graphs $H\in O(G)$ such that $s(H)=\{x\}$.
\end{itemize}\end{notation}

Let us start by a combinatorial lemma.

\begin{prop}
Let $G$ be a graph and $x,y\in V(G)$. Then $O(G,x)$ and $O(G,y)$ are in bijection.
\end{prop}

\begin{proof}
We assume that $x\neq y$. 
Let $H\in O(G)$. We define a partial order $\leq_H$ on $V(G)$ such that for any $x,y\in V(G)$,
$x\leq_H y$ if there exists an oriented path $(x=x_1,x_2,\ldots,x_k=y)$ in $H$. As $H$ is acyclic, this relation
is antisymmetric. It is obviously reflexive and transitive, so it is an order on $V(G)$. The set of minimal elements
of $(V(G),\leq_H)$ is $s(H)$.

Let $H\in O(G,x)$. As $s(H)=\{x\}$, $x$ is the unique minimal element of $(V(G),\leq_H)$, so $x\leq_H y$.
We consider
\[[x,y]_H=\{z\in V(G)\mid x \leq_H z\leq_H y\}.\]
This is non empty. Let $H'$ be the oriented graph obtained by changing the orientations of all the edges
between two vertices of $[x,y]_H$. 
Let $(x_1,\ldots,x_k=x_1)$ be a cycle in $H'$. 
\begin{itemize}
\item If none of the vertices of this cycle belongs to $[x,y]_H$, then it is a cycle in $H$: this is contradiction, as $H$ is acyclic. 
\item Let us assume that at least one of the vertices of this cycle belongs to $[x,y]_H$: up to a permutation, we assume that $x_1=x_k\in [x,y]_H$.
Let us prove by induction on $i$ that $x\leq_H x_i$ for any $i$. It is obvious if $i=1$. Let us assume that $x_{i-1}\leq y$. Two cases are possible:
\begin{itemize}
\item If $(x_{i-1},x_i)\in A(H)$, then $x\leq_H x_{i-1}\leq_H x_i$, so $x\leq_H x_i$.  
\item If $(x_i,x_{i-1})\in A(H)$, by definition of $H'$, this implies that $x_i,x_{i-1}\in [x,y]_H$, so $x\leq_H x_i$.
\end{itemize} 
Let us now prove by induction on $i$ that $x_{k-1}\leq_H y$ for any $i$. It is obvious for $i=0$, as $x_k=x_1$. 
Let us assume that $x_{k+1-i}\leq_H y$. Two cases are possible.
\begin{itemize}
\item If $(x_{k-i},x_{k+1-i})\in A(H)$, then $x_{k-i}\leq_H x_{k+1-i}\leq_H y$, so $x_{k-i}\leq_H y$.  
\item If $(x_{k+1-i},x_{k-i})\in A(H)$, by definition of $H'$, this implies that $x_{k+1-i},x_{k-i}\in [x,y]_H$, so $x_{k-i}\leq_H y$.
\end{itemize} 
We obtain that $x_1,\ldots,x_k\in [x,y]_H$, so $(x_k,x_{k-1},\ldots,x_1)$ is a cycle of $H$: this is a contradiction, as $H$ is acyclic.
\end{itemize}
As a conclusion, $H'$ is acyclic. \\

Let $z\in V(H')$. If $z\notin [x,y]_H$, then it is not a source of $H$ (as the unique source of $H$ is $x$,)
so there exists $t\in V(H)$, such that $(t,z)\in A(H)$. Then $(t,z)\in A(H')$ and $z\notin s(H')$. 
Let $z\in [x,y]_H$, different from $y$. As $z<_H y$, there exists a walk in $H$ from $z$ to $y$,
so there exists $t\in [x,y]_H$ such that $(z,t)\in A(H)$. Then $(t,z)\in A(H')$, so $z\notin s(H')$.
Finally, $s(H')\subseteq \{y\}$ and, as $s(H')\neq \emptyset$, $s(H')=\{y\}$. We proved that
$H'\in O(G,y)$. This define a map
\[f_{x,y}:\left\{\begin{array}{rcl}
O(G,x)&\longrightarrow&O(G,y)\\
H&\longrightarrow&H'.
\end{array}\right.\]

Let us consider $[y,x]_{H'}$. By construction of $H'$, $[x,y]_H\subseteq [y,x]_{H'}$. 
Let $z\in [y,x]_{H'}$. There exists a walk $(x_0=y,\ldots,x_j=z,\ldots,x_k=x)$ in $H'$. 
Let us prove by induction on $i$ that $x\leq_H x_i$ for any $i$. It is obvious if $i=0$, as $x\leq_H y$.
Let us assume that $x\leq_H x_{i-1}$. Two cases are possible.
\begin{itemize}
\item If $(x_{i-1},x_i)\in A(H)$, then $x\leq_H x_{i-1}\leq_H x_i$, so $x\leq_H x_i$.
\item If $(x_i,x_{i-1})\in A(H)$, then by definition of $H'$, $x_{i-1},x_i\in [x,y]_H$, so $x\leq_H x_i$.
\end{itemize} 
Let us now prove by induction on $i$ that $x_{k-i}\leq_H y$ for any $i$. It is obvious if $k=0$, as $x\leq_H y$. Let us assume that $x_{k+1-i}\leq_H y$. 
Two cases are possible.
\begin{itemize}
\item If $(x_{k-i},x_{k+1-i})\in A(H)$, then $x_{k-i}\leq_H x_{k+1-i}\leq_H y$, so $x_{k-i}\leq_H y$.
\item If $(x_{k+1-i},x_{k-i})\in A(H)$, then by definition of $H'$, $x_{k+1-i},x_{k-i}\in [x,y]_H$, so $x_{k-i}\leq_H y$.
\end{itemize} 
We proved that any vertex of $(x_0,\ldots,x_k)$ belongs to $[x,y]_H$. 
In particular, $z\in [x,y]_H$. Therefore, $[x,y]_H=[y,x]_{H'}$. As a consequence, 
\[f_{y,x}\circ f_{x,y}=\id_{O(G,x)}. \]
So $f_{x,y}$ is a bijection for any $x\neq y\in V(G)$, of inverse $f_{y,x}$. 
\end{proof}

Consequently, we define:

\begin{defi}
For any graph $G$, choosing any vertex $x\in V(G)$,  we denote by $\tilde{\phi}(G)$ the number of 
acyclic orientations of $G$, such that $s(G)=\{x\}$. By convention, $\tilde{\phi}(1)=0$. 
This defines an infinitesimal character of $G$.
\end{defi}

\begin{proof}
By the preceding lemma, this does not depend on the choice of $x$. As any non connected graph $G$ has at least
two sources, if $G$ is not connected, then $\tilde{\phi}(G)=0$. So $\tilde{\phi}$ is an infinitesimal character.
\end{proof}

Here is a second combinatorial lemma:

\begin{lemma}
Let $G$ be a graph and $e\in E(G)$. We denote by $G/e$ the graph obtained by contraction of $e$
(and so identifying the two extremities of $e$) and by $G\setminus e$ the graph obtained by deletion of $e$.
Then
\begin{align*}
\Phi_{chr}(G)&=\Phi_{chr}(G\setminus e)-\Phi_{chr}(G/e),\\
\phi_{chr}(G)&=\phi_{chr}(G\setminus e)-\phi_{chr}(G/e),\\
\tilde{\phi}(G)&=\tilde{\phi}(G\setminus e)+\tilde{\phi}(G/e).
\end{align*}
\end{lemma}

\begin{proof}
We put $e=\{x,y\}$. 

1. Let us give a proof of this classical result. 
Let $n\in \N$. Then $\Phi_{chr}(G\setminus e)(n)$ is the numbers of colorations $f$ of $G$ such that 
for any $e'=\{x',y'\}\in V(G)$, $e'\neq e$, $f(x')\neq f(y')$. 
Moreover, $\Phi_{chr}(G/e)(n)$ is the numbers of colorations $f$ of $G$ such that 
for any $e'=\{x',y'\}\in V(G)$, $e'\neq e$, $f(x')\neq f(y')$, and such that $f(x)=f(y)$. Taking the difference,
$\Phi_{chr}(G\setminus e)(n)-\Phi_{chr}(G/e)(n)=\Phi_{chr}(G)(n)$ for any $n\in \N$,
which gives the first equality.\\

2. Direct consequence of the first point, as $\phi_{chr}(H)=\Phi_{chr}(H)'(0)$ for any graph $H$.\\

3. Let us denote by $\overline{x}$ the vertex of $G/e$ obtained by identification of $x$ and $y$. 
If $H$ is an orientation of $G$, we denote by $H'$ the orientation of $G$ obtained from $H$
by changing the sense of $e$, and we put
\begin{align*}
O_1(G,x)&=\{H\in O(G,x)\mid H'\notin O(G,y)\},\\
O_2(G,x)&=\{H\in O(G,x)\mid H'\in O(G,y)\}.
\end{align*}
Let $\overline{H}$ be an orientation of $G/e$ and let $H_1$ and $H_2$ be the two orientations of $G$
inducing $\overline{H}$: in $H_1$, $e$ is oriented from $x$ to $y$ whereas in $H_2$, it is oriented from $y$ to $x$.
We assume that $\overline{H}\in O(G/e,\overline{x})$. As $\overline{H}$ is acyclic, $H_1$ and $H_2$ are acyclic
(as the contraction of a cycle is a cycle). Let $z$ be a source of $H_1$ or of $H_2$. If $z\neq x,y$,
it is also a vertex of $G/e$ and is not a source of $G/e$, so it is not a source of $H_1$, nor of $H_2$.
By construction, $y$ is not a source of $H_1$ and $x$ is not a source of $H_2$. 
Hence, $H_1\in O(G,x)$ and $H_2\in O(G,y)$: we obtain that $H_1\in O_2(G,x)$.  We obtain in this way a bijection
from $O(G/e,\overline{x})$ to $O_2(G,x)$, so 
\[\overline{\phi}(G/e)=|O_2(G,x)|.\]

Let $H\in O(G\setminus e,x)$ and let $H_1$ and $H_2$ be the two orientations of $G$
inducing $H$: in $H_1$, $e$ is oriented from $x$ to $y$ whereas in $H_2$, it is oriented from $y$ to $x$.
As $y$ is not a source of $H$, it is not a source of $H_2$, so $H_2\notin O(G,y)$. 
As $x$ is the unique source of $H$, $x$ is the unique source of $H_1$. If $(x_1,\ldots,x_k=x_1)$
is a cycle of $H_1$, then necessarily $e$ is one of the walks $(x_i,x_{i+1})$, as $H$ has no cycle:
this is not possible, as $x$ is a source in $H_1$. So $H_1\in O(G,x)$ and finally $H_1\in O_1(G,x)$. 
We obtain in this way a bijection from $O(G\setminus e)$ to $|O_1(G,x)|$, so
\[\overline{\phi}(G\setminus e)=|O_1(G,x)|.\]
Summing, this gives the announced formula. 
\end{proof}

The following result is firstly due to Greene and Zaslavsky \cite{Greene1983}, see \cite{Gebhard2000} for several proofs of different natures:

\begin{theo} \label{theo4.9}
For any graph $G$, $\phi_{chr}(G)=(-1)^{|V(G)|+1}\tilde{\phi}(G)$.
\end{theo}

\begin{proof}
We proceed by induction on the number $n$ of edges of $G$. If $E(G)=\emptyset$,
then $G=\grun^n$ for a certain $n\in \N$. Then $\Phi_{chr}(G)=X^n$, so
\[\phi_{chr}(G)=\tilde{\phi}(G)=\begin{cases}
1\mbox{ if }n=1,\\
0\mbox{ otherwise}.
\end{cases}\]
Let us assume the result at all ranks $<n$. Let us choose any edge $e$ of $G$.
As $G/e$ and $G\setminus e$ has strictly less than $n$ edges,
\begin{align*}
\phi_{chr}(G)&=\phi_{chr}(G\setminus e)-\phi_{chr}(G/e)\\
&=(-1)^{|V(G)|+1}\overline{\phi}(G\setminus e)-(-1)^{|V(G)|}\overline{\phi}(G/e)\\
&=(-1)^{|V(G)|+1}(\overline{\phi}(G\setminus e)+\overline{\phi}(G/e))\\
&=(-1)^{|V(G)|+1}\overline{\phi}(G). \qedhere
\end{align*}
\end{proof}

\subsection{Generalization to commutative connected bialgebras}

We proved in Proposition \ref{prop4.4} that in the case of a connected double bialgebra,
the eulerian idempotent $\varpi$ is a projector. We now extend this result to any commutative connected bialgebra.

\begin{lemma}
Let $(A,m,\Delta)$ be a commutative or cocommutative bialgebra. 
The induced convolution product on $\en(A)$ is denoted by $*$. The canonical projection on the augmentation ideal
of $A$ is denoted by $\rho$. There exists a family of scalars $(\lambda(k,l,p))_{k,l,p\in\N}$, 
which does not depend on $A$, such that for any $k,l\in \N$,
\[\rho^{*k}\circ \rho^{*l}=\rho^{*l}\circ \rho^{*k}=\sum_{p=0}^{kl} \lambda(k,l,p)\rho^{*p}.\]
\end{lemma}

\begin{proof}
We shall use Sweedler's notation for $\Delta(x)=x^{(1)}\otimes x^{(2)}$ for any $x\in A$. Let $x\in A$. Then 
\[\id^{*k}(x)=x^{(1)}\ldots x^{(k)}.\]
Therefore, for any $x\in A$,
\begin{align*}
\id^{*k}\circ \id^{*l}(x)&=\id^{*k}\left(x^{(1)}\ldots x^{(l)}\right)\\
&=\left(x^{(1)}\ldots x^{(l)}\right)^{(1)}\ldots \left(x^{(1)}\ldots x^{(l)}\right)^{(k)}\\
&=x^{(1)}x^{(l+1)}\ldots x^{((k-1)l+1)}\ldots x^{(k)}x^{(2k)}\ldots x^{(kl)}\\
&=x^{(1)}\ldots x^{(kl)}\\
&=\id^{*kl}(x).
\end{align*}
We use that $A$ is commutative or cocommutative for the fourth equality. Hence, $\id^{*k}\circ \id^{*l}=\id^{*kl}$. \\

Let $\iota$ be the unit of $*$. Then $\rho=\id-\iota$, and
\begin{align*}
\rho^{*k}\circ \rho^{*l}&=(\id-\iota)^{*k}\circ (\id-\iota)^{*l}\\
&=\sum_{i=0}^k \sum_{j=0}^l (-1)^{i+j}\binom{k}{i}\binom{l}{j} \id^{*i}\circ \id^{*j}\\
&=\sum_{i=0}^k \sum_{j=0}^l (-1)^{i+j}\binom{k}{i}\binom{l}{j} \id^{*ij}\\
&=\sum_{i=0}^k \sum_{j=0}^l (-1)^{i+j}\binom{k}{i}\binom{l}{j} (\rho+\iota)^{*ij}\\
&=\sum_{p=0}^\infty\underbrace{\left(\sum_{i=0}^k\sum_{j=0}^l (-1)^{i+j}
\binom{k}{i}\binom{l}{j}\binom{ij}{p}\right)}_{=\lambda(k,l,p)}\rho^{*p}.
\end{align*}
Note that $\lambda(k,l,p)=0$ if $p>kl$. 
As for any $k,l,p\in \N$, $\lambda(k,l,p)=\lambda(l,k,p)$, $\rho^{*k}\circ \rho^{*l}=\rho^{*l}\circ \rho^{*k}$. 
\end{proof}

\begin{lemma} \label{lemma4.11}
Let $f(T)\in \K[[T]]$ and let $\rho$ be the projection on the augmentation ideal of $\K[X]$.
If $f(\rho)=0$, then $f=0$. 
\end{lemma}

\begin{proof}
For any $k,n\in \N$,
\[\rho^{*k}(X^n)=\left(\sum_{\substack{i_1+\ldots+i_k=n,\\i_1,\ldots,i_k\geq 1}}\frac{n!}{i_1!\ldots i_k!}\right)X^n.\]
In particular, $\rho^{*k}(X^k)=k!X^k \neq 0$ and $\rho^{*k}(X^n)=0$ if $n<k$.
Let $f\in \K[[T]]$, nonzero, and let $k=\val(f)$. Then
\[f(\rho)(X^k)=a_k \rho^{*k}(X^k)+0=a_k k!X^k\neq 0,\]
so $f(\rho)\neq 0$. 
\end{proof}

\begin{lemma}
Let $p,k,l\in \N$. If $p<k$ or $p<l$, then $\lambda(k,l,p)=0$. 
\end{lemma}

\begin{proof}
We work in the bialgebra $(\K[X],m,\Delta)$. If $p<l$, then $\rho^{*k}\circ \rho^{*l}(X^p)=0$, 
as $\rho^{*l}(X^p)=0$. We consider the formal series
\[f(T)=\sum_{p=0}^{kl}\lambda(k,l,p)T^p\in \K[[T]].\]
Then $f(\rho)=T^{*k}\circ T^{*l}$. Let $q=\val(f)$. Then 
\[f(\rho)(T^q)=\lambda(k,l,q)q!X^q \neq 0,\]
so $q\geqslant l$. By symmetry in $k,l$ of the coefficients $\lambda(k,l,p)$, $\val(f)\geqslant k$. 
\end{proof}

\begin{prop}
 \label{prop4.13}
Let $A$ be a connected commutative bialgebra. We put
\[\varpi=\ln(\id)=\sum_{k=1}^\infty \frac{(-1)^{k+1}}{k}\rho^{*k}.\]
Then $\varpi$ is a projection. Its kernel is $A_+^2\oplus \K1_A$ and its image contains $\prim(A)$.
\end{prop}

\begin{proof}
By definition of the coefficients $\lambda(k,l,p)$ and by the preceding lemma, for any formal series
$f=\sum a_k T^k$ and $g=\sum b_kT^k$ in $\K[[T]]$,
\[f(\rho)\circ g(\rho)=\sum_{p=0}^\infty \left(\sum_{k,l\leqslant p} \lambda(k,l,p)a_kb_l\right)\rho^{*p}.\]

We consider the case where $A=(\K[X],m,\Delta)$. As it is a double bialgebra, in this case,
by Proposition \ref{prop4.4}, $\varpi$ is a projection. Hence, 
\begin{align*}
\varpi\circ \varpi&=\sum_{p=1}^\infty \left(\sum_{1\leqslant k,l\leqslant p} \lambda(k,l,p)\frac{(-1)^{k+l}}{kl}
\right)\rho^{*p}\\
&=\varpi\\
&=\sum_{p=1}^\infty \frac{(-1)^{p+1}}{p} \rho^{*p}.
\end{align*}
By Lemma \ref{lemma4.11}, for any $p\in \N^*$, 
\[\sum_{k,l\leqslant p} \lambda(k,l,p)\frac{(-1)^{k+l}}{kl}=\frac{(-1)^{p+1}}{p}.\]

Let us now turn to the general case. 
\begin{align*}
\varpi\circ \varpi&=\sum_{p=1}^\infty \left(\sum_{1\leqslant k,l\leqslant p} \lambda(k,l,p)\frac{(-1)^{k+l}}{kl}
\right)\rho^{*p}=\sum_{p=1}^\infty \frac{(-1)^{p+1}}{p} \rho^{*p}=\varpi,
\end{align*}
so $\varpi$ is a projection.\\

Let $x\in \prim(A)$. Then $\varpi(x)=\rho(x)+0=x$, so $x\in \im(\varpi)$. Let $x\in \ker(\varpi)\cap A_+$. Then
\[\rho(x)=0=x+\underbrace{\sum_{k=2}^\infty \frac{(-1)^{k+1}}{k}\rho^{*k}(x)}_{\in A_+^2},\]
so $x\in A_+^2$. We obtain that $\ker(\varpi)\subseteq A_+^2\oplus\K1_A$. Note that $\varpi(1_A)=0$. Moreover,
\begin{align*}
\pi\circ m&=m^*(\pi)\\
&=m^*(\ln(\id))\\
&=\ln(m^*(\id))\\
&=\ln(\id \otimes \id)\\
&=\ln((\id \otimes \iota)*(\iota\otimes \id))\\
&=\ln(\id \otimes \iota)+\ln(\iota \otimes \id)\\
&=\ln(\id)\otimes \iota+\iota \otimes \ln(\id)\\
&=\varpi\otimes \iota+\iota \otimes \varpi.
\end{align*}
Therefore, if $x,y\in A_+$,
\[\varpi(xy)=\varpi(x)\varepsilon(y)+\varepsilon(x)\varpi(y)=0.\]
So $A_+^2\oplus \K1_A\subseteq \ker(\varpi)$. 
\end{proof}

\begin{cor} \label{cor4.14}
Let $(A,m,\Delta)$ be a connected and commutative bialgebra. 
Then $(A,m,\Delta)$ is isomorphic to a subbialgebra of the shuffle algebra $(T(\prim(A)),\shuffle,\Delta)$.
\end{cor}

\begin{proof}
By the universal property of $(T(\prim(A)),\shuffle,\Delta)$, (Proposition \ref{prop3.5}), 
for any linear map $\phi:A\longrightarrow \prim(A)$ such that $\phi(1_A)=0$, 
there exists a unique coalgebra morphism $\Phi:A\longrightarrow T(\prim(A))$
such that $\pi\circ \Phi=\phi$, where $\Phi:T(\prim(A))\longrightarrow \prim(A)$ is the canonical projection.

By Proposition \ref{prop4.13}, $\prim(A)\cap A_+^2=(0)$. Let us choose $\phi$ such that 
$\phi_{\mid \prim(A))}=\id_{\prim(A)}$ and $\phi(A_+^2)=(0)$. We denote by $\Phi$ the corresponding 
coalgebra morphism from $A$ to $T(\prim(A))$. As $\Phi_{\mid \prim(A)}$ is injective, by Proposition \ref{prop3.5},
$\Phi$ is injective. As $\phi(A_+^2)=(0)$, still by Proposition \ref{prop3.5}, $\Phi$ is a bialgebra morphism
from $(A,m,\Delta)$ to $(T(\prim(A)),\shuffle,\Delta)$. 
\end{proof}

\begin{cor} \label{cor4.15}
Let $(A,m,\Delta)$ be a connected commutative bialgebra. Then it can be embedded in a double bialgebra
$(B,m,\Delta,\delta)$, with $\prim(B)=\prim(A)$. 
If $A$ is cofree or if $A$ is cocommutative, then there exists a second coproduct $\delta$ on $A$
making it a double bialgebra.
\end{cor}

\begin{proof}
\textit{First step}. 
Let $(V,\cdot)$ be a commutative algebra. We can consider the quasishuffle algebra $(T(V),\squplus,\Delta)$. 
By Corollary \ref{cor4.15} and its proof, choosing a convenient $\phi$, 
there exists an injective bialgebra morphism $\Phi:(T(V),\squplus,\Delta)\longrightarrow (T(V),\shuffle,\Delta)$,
such that $\Phi(v)=v$ for any $v\in V$. Moreover, for any $v_1,\ldots,v_n \in V$, with $n\geq 1$,
\begin{align*}
\Phi(v_1\ldots v_n)&=\frac{1}{1-\phi}(v_1\ldots v_n)\\
&=\sum_{k=1}^n \sum_{\substack{v_1\ldots v_n=w_1\ldots w_k,\\ w_1,\ldots,w_i\neq 1}}
\underbrace{\phi(w_1)\ldots \phi(w_k)}_{\in V^{\otimes k}}
=v_1\ldots v_n+\mbox{words of length $<n$}.
\end{align*}
An easy triangularity argument proves then that $\Phi$ is bijective. Hence, $(T(V),\squplus,\Delta)$
and $(T(V),\shuffle,\Delta)$ are isomorphic. 

In the particular case where $V$ is a commutative bialgebra, 
we obtain a second coproduct $\delta$ on $T(V)$, making it a double bialgebra.\\

\textit{Second step}. Let $(A,m,\Delta)$ be a connected bialgebra. Let us choose any commutative bialgebra
structure on $\prim(A)$. As $(T(\prim(A)),\shuffle,\Delta)$ and $(T(\prim(A)),\squplus,\Delta)$ are isomorphic,
from Corollary \ref{cor4.14}, there exists an injective bialgebra morphism from $A$
to $(T(\prim(A)),\squplus,\Delta)$, which proves the first point, as $(T(\prim(A)),\squplus,\Delta)$
is a double bialgebra. \\

\textit{Last step}. If $A$ is cofree, then the injection from $A$ to $T(\prim(A))$ is a bijection.
If $A$ is cocommutative, as it is connected it is primitively generated by Cartier-Quillen-Milnor-Moore's theorem:
hence, its image is the subalgebra $A'$ of $(T(V),\squplus)$ generated by $\prim(T(V))=V$. 
As $\delta(V)\subseteq V\otimes V$ by construction of $V$, $\delta(A')\subseteq A'\otimes A'$
so $A'$ is a double subbialgebra of $(T(\prim(A)),\squplus, \Delta,\delta)$. 
\end{proof}

\subsection{Antipode and eulerian idempotent for quasishuffle algebras}

\begin{notation}
\begin{enumerate}
\item We identify $\K[[T]]$ and the dual of $\K[X]$, with the pairing defined by
\[\langle \sum_{k=0}^\infty a_kT^k,\sum_{k=0}^N b_k X^k\rangle=\sum_{k=0}^N a_kb_k.\]
\item Let $g:[n]\twoheadrightarrow [l]$ be a surjective map. We put
\begin{align*}
d(g)&=\sharp\{i\in [n-1]\mid g(i)\geq g(i+1)\},\\
P_g(X)&=X^{d(g)+1}(1+X)^{n-1-d(g)} \in \K[X].
\end{align*}
The letter $d$ is for \emph{descents}. 
\end{enumerate}

\end{notation}

\begin{prop} \label{prop4.16}
Let $(V,\cdot,\delta_V)$ be a commutative, non necessarily unitary bialgebra. 
We consider the double quasishuffle bialgebra $(T(V),\squplus,\Delta,\delta)$. 
Let $Q(T)\in \K[[T]]$ and let $\lambda=Q(\epsilon_\delta-\varepsilon_\Delta) \in T(V)^*$. For any
word $v_1\ldots v_n \in V^{\otimes n}$, with $n\geq 1$,
\[\Theta(\lambda)(v_1\ldots v_n)=\sum_{l=1}^n\sum_{g:[n]\twoheadrightarrow [l]} \langle Q(T),P_g(X)\rangle 
\left(\prod_{g(i)=1}^\cdot v_i\right)\ldots \left(\prod_{g(i)=l}^\cdot v_i\right). \]
\end{prop}

\begin{proof}
For any $v_1\ldots v_n \in V^{\otimes n}$, with $n\geq 1$,
\begin{align*}
\lambda(v_1\ldots v_n)&=Q(\epsilon_\delta-\varepsilon_\Delta)(v_1\ldots v_n)
=\sum_{k=1}^\infty a_k \epsilon_\delta^{\otimes k}\circ \tdelta^{(k-1)}(v_1\ldots v_n)
=a_n \epsilon_V(v_1)\ldots \epsilon_V(v_n),
\end{align*}
as $\epsilon_\delta$ vanishes on any word of length $\geq 2$. 
By definition of the coproduct $\delta$,
\begin{align*}
&\delta(v_1\ldots v_n)\\
&=\sum_{\substack{k,l\geq 1,\\ f:[n]\twoheadrightarrow [k],\mbox{\scriptsize increasing}\\
g:[n]\twoheadrightarrow [l],\\
\forall i,j\in [n],\: (i<j\:\mbox{\scriptsize and }f(i)=f(j))\Longrightarrow g(i)<g(j)
}}\left(\prod_{f(i)=1}v'_i\right)\ldots \left(\prod_{f(i)=k}v'_i\right)\otimes
\left(\prod_{g(i)=1}v''_i\right)\ldots \left(\prod_{g(i)=l}v''_i\right)\\
&=\sum_{\substack{k,l\geq 1,\\g:[n]\twoheadrightarrow [l],\\
f:[n]\twoheadrightarrow [k],\mbox{\scriptsize increasing},\\
\forall i,j\in [n],\: (i<j\:\mbox{\scriptsize and }g(i)\geq g(j))\Longrightarrow f(i)<f(j)
}}\left(\prod_{f(i)=1}v'_i\right)\ldots \left(\prod_{f(i)=k}v'_i\right)\otimes
\left(\prod_{g(i)=1}v''_i\right)\ldots \left(\prod_{g(i)=l}v''_i\right).
\end{align*}
For any $g:[n]\twoheadrightarrow [l]$, let us put
\[A(g)=\{f:[n]\twoheadrightarrow[k],\mbox{ increasing}\mid \forall i,j\in [n],\:
(i<j\mbox{ and }g(i)\geq g(j))\Longrightarrow f(i)<f(j)\}.\]
Then, putting $Q(T)=\sum a_kT^k$,
\begin{align*}
\Theta(\lambda)(v_1\ldots v_n)&=(\lambda \otimes \id)\circ \delta(v_1\ldots v_n)\\
&=\sum_{l\geq 1}\sum_{g:[n]\twoheadrightarrow [l]}\sum_{f\in A(g)}
\lambda\left(\left(\prod_{f(i)=1}v'_i\right)\ldots \left(\prod_{f(i)=\max(f)}v'_i\right)
\right)\left(\prod_{g(i)=1}v''_i\right)\ldots \left(\prod_{g(i)=l}v''_i\right)\\
&=\sum_{l\geq 1}\sum_{g:[n]\twoheadrightarrow [l]}\sum_{f\in A(g)} a_{\max(f)}
\prod_{i=1}^n \epsilon_V(v'_i)\left(\prod_{g(i)=1}v''_i\right)\ldots \left(\prod_{g(i)=l}v''_i\right)\\
&=\sum_{l\geq 1}\sum_{g:[n]\twoheadrightarrow [l]}\left(\sum_{f\in A(g)} a_{\max(f)}\right)
\left(\prod_{g(i)=1}v_i\right)\ldots \left(\prod_{g(i)=l}v_i\right).
\end{align*}
For any $k\in \N_{>0}$, we put
\[A_k(g)=\{f\in A(g)\mid \max(f)=k\}\]
and we put
\[R_g(X)=\sum_{k\geq 1} |A_k(g)|X^k.\]
With this definition, we obtain that
\[\Theta(v_1\ldots v_n)=\sum_{l\geq 1}\sum_{g:[n]\twoheadrightarrow [l]} \langle Q(T),R_g(X)\rangle 
\left(\prod_{g(i)=1}^\cdot v_i\right)\ldots \left(\prod_{g(i)=l}^\cdot v_i\right). \]

It remains to prove that $R_g(X)=P_g(X)$. For this, let us now study $A(g)$ for $g:[n]\twoheadrightarrow [l]$. 
For any  $f:[n-1]\longrightarrow [k]$, increasing, we put
\begin{align*}
\upsilon_0(f):&\left\{\begin{array}{rcl}
[n]\longrightarrow&[k]\\
i\in [n-1]&\longrightarrow&f(i),\\
n&\longrightarrow&f(n-1),
\end{array}\right.&
\upsilon_+(f):&\left\{\begin{array}{rcl}
[n]\longrightarrow&[k]\\
i\in [n-1]&\longrightarrow&f(i),\\
n&\longrightarrow&f(n-1)+1.
\end{array}\right.
\end{align*}
Denoting by $g'$ the standardization of the restriction of $g$ to $[n-1]$ (that is to say the composition of $g_{\mid [n-1]}$
with the unique increasing bijection from $g([n-1])$ to $[l']$ for a well-chosen $l'$), we obtain that
\begin{align*}
A(g)&=\begin{cases}
\left\{\upsilon_+(f)\mid f\in A(g')\right\}\mbox{ if }g(n-1)\geq g(n),\\
\left\{\upsilon_+(f),\:\upsilon_0(f)\mid f\in A(g')\right\}\mbox{ if }g(n-1)<g(n).
\end{cases}
\end{align*}
As $\upsilon_0$ does not change the maximum and $\upsilon_1$ increases it by $1$:
\begin{itemize}
\item If $g(n-1)\geq g(n)$, then $d(g)=d(g')+1$ and $|A_k(g)|=|A_{k-1}(g')|$. Hence, $R_g(X)=XR_{g'}(X)$.
\item If $g(n-1)>g(n)$, then $d(g)=d(g')$ and $|A_k(g)|=|A_k(g')|+|A_{k-1}(g')|$. Hence, $R_g(X)=(X+1)R_{g'}(X)$. 
\end{itemize}
The result $R_g(X)=P_g(X)$ then comes from an easy induction on $n$.
\end{proof}

Let us apply this formula for $Q(T)=\dfrac{1}{1+T}$ and $Q(T)=\ln(1+T)$:

\begin{cor} \label{cor4.17}
Let $(V,\cdot)$ be a commutative (non necessarily unitary) algebra. In the Hopf algebra $(T(V),\squplus, \Delta)$,
the antipode $S$ is given on any non empty word $v_1\ldots v_n$ by
\[S(v_1\ldots v_n)=(-1)^n\sum_{l\geq 1}\sum_{g:[n]\twoheadrightarrow[l],\mbox{\scriptsize decreasing}}
\left(\prod_{g(i)=1}^\cdot v_i\right)\ldots \left(\prod_{g(i)=l}^\cdot v_i\right).\]
The eulerian idempotent is given on any non empty word $v_1\ldots v_n$ by
\[\varpi(v_1\ldots v_n)=\sum_{l\geq 1}\sum_{g:[n]\twoheadrightarrow[l]}
(-1)^{d(g)}\dfrac{d(g)! (n-1-d(g))!}{n!}
\left(\prod_{g(i)=1}^\cdot v_i\right)\ldots \left(\prod_{g(i)=l}^\cdot v_i\right).\]
\end{cor}

\begin{proof}
By functoriality, as any commutative algebra is the quotient of a commutative bialgebra, it is enough to prove it
for a commutative bialgebra. For the antipode, we use Proposition \ref{prop4.16} with $Q(T)=\dfrac{1}{1+T}$. 
For any $n\in \N$,
\[\langle Q(T),X^n\rangle=(-1)^n,\]
so $\langle Q(T),P(X)\rangle=P(-1)$ for any $P(X)\in \K[X]$. Therefore, if $g:[n]\twoheadrightarrow [l]$,
\begin{align*}
\langle Q(T),P_g(X)\rangle=P_g(-1)&=\begin{cases}
(-1)^n \mbox{ if }d(g)=n-1,\\
0\mbox{ otherwise},
\end{cases}\\
&=\begin{cases}
(-1)^n \mbox{ if $g$ is decreasing},\\
0\mbox{ otherwise}.
\end{cases}
\end{align*}
For the eulerian idempotent, we use $Q(T)=\ln(1+T)$. For any $n\in \N_{>0}$,
\[\langle Q(T),X^n\rangle=\dfrac{(-1)^{n+1}}{n}=\int_{-1}^0 t^{n-1}\mathrm{d}t,\]
so for any $P(X)\in X\K[X]$,
\[\langle Q(T),P(X)\rangle=\int_{-1}^0 \dfrac{P(t)}{t}\mathrm{d}t.\]
In particular, if $g:[n]\twoheadrightarrow [l]$,
\[\langle Q(T),P_g(X)\rangle=\int_{-1}^0 t^{d(g)}(1+t)^{n-1-d(g)}\mathrm{d}t.\]
An easy induction on $p$, based on an integration by part, proves that for any $p,q\in \N$,
\[\int_{-1}^0 t^p(1+t)^q\mathrm{d}t=(-1)^k\dfrac{k!l!}{(k+l+1)!}.\]
The result immediately follows, with $p=d(g)$ and $q=n-1-d(g)$. \end{proof}

\section{Graded double bialgebras}

\subsection{Reminders}

\begin{defi}\label{defi5.1}
Let $(B,m,\Delta)$ be a bialgebra. We shall say that it is graded if there exists a graduation $(B_n)_{n\in \N}$
of $B$ such that:
\begin{itemize}
\item For any $k,l\in \N$, $m(B_k\otimes B_l)\subseteq B_{k+l}$.
\item For any $n\in \N$, $\displaystyle \Delta(B_n)\subseteq \sum_{k=0}^n B_k\otimes B_{n-k}$.
\end{itemize}
We shall say that the graduation is connected if $B_0=\K1_B$.
\end{defi}

\begin{example}
This is the case of $\K[X]$, with $\K[X]_n=\K X^n$ for any $n$. The bialgebra $\QSym$ is also graded and connected,
putting any composition $(k_1\ldots k_n)$ homogeneous of degree $k_1+\ldots+k_n$. 
The bialgebra of graphs $(\calH_\gr,m,\Delta)$ is also graded by the number of vertices of the graphs.
\end{example}

\begin{remark}
If $(B,m,\Delta)$ is a graded and connected bialgebra, then for any $n\geq 1$,
\[\tdelta(B_n)\subseteq \sum_{k=1}^{n-1} B_k\otimes B_{n-k}.\]
Inductively, for any $n,k\geq 1$,
\[\tdelta^{(k-1)}(B_n)\subseteq \sum_{\substack{i_1+\ldots +i_k=n,\\i_1,\ldots,i_k\geq 1}}B_{i_1}\otimes
\ldots \otimes B_{i_k}.\]
In particular, if $x\in B_n$, for any $k>n$, $\tdelta^{(k-1)}(x)=0$: $B$ is connected in the sense of the preceding section.
\end{remark}

\subsection{Homogeneous polynomial invariants}

If $(B,m,\Delta,\delta)$ is a graded connected bialgebra, a natural question is to find all the homogeneous bialgebra
morphisms from $B$ to $\K[X]$. For this, we identify $B_1^*$ with
\[\{\lambda\in B^*\mid \forall n\neq 1, \:\lambda(B_n)=(0)\}.\]
Note that as $B$ is connected, $B_1^*\subseteq \infchara(B)$.  We obtain:

\begin{prop}\label{prop5.2}
Let $\mu\in \infchara(A)$ and let $\Psi_\mu:(B,m,\Delta)\longrightarrow (\K[X],m,\Delta)$
associated to $\mu$ by Proposition \ref{prop3.10}. Then $\Psi_\mu$ is homogeneous if, and only if,
$\mu\in B_1^*$.
\end{prop}

\begin{proof}
$\Longleftarrow$. Let us assume that $\mu\in B_1^*$. Let $n\geq 1$ and $x\in B_n$. Then 
\[\Psi_\mu(x)=\sum_{k=1}^\infty \mu^{\otimes k}\circ \tdelta^{(k-1)}(x)\frac{X^k}{k!}=
\mu^{\otimes n}\circ \tdelta^{(n-1)}(x)\frac{X^n}{n!},\]
as $\tdelta^{(k-1)}(x)$ has no component in $B_1^{\otimes k}$ if $k\neq n$ by homogeneity of $\tdelta$.
Therefore, $\Psi_\mu$ is homogeneous. \\

$\Longrightarrow$. Let us assume that $\mu\notin B_1^*$. As $\mu(1_B)=0$, there exists $n\geq 2$
and $x\in B_n$ such that $\mu(x)\neq 0$. Then the coefficient of $X$ in $\Psi_\mu(x)$ is $\mu(x)\neq 0$,
so $\Psi_\mu(x)$ is not homogeneous of degree $n$ and $\Psi_\mu$ is not homogeneous. 
\end{proof}

\begin{cor}\label{cor5.3}
Let $\mu\in B_1^*$. For any $x\in B_n$, with $n\geq 1$,
\begin{align*}
\exp(\mu)(x)&=\frac{1}{n!}\mu^{\otimes n}\circ \tdelta^{(n-1)}(x),&
\Psi_\mu(x)&=\exp(\mu)(x)X^n.
\end{align*}
\end{cor}

\begin{proof} 
By homogeneity of $\mu$,
\[\exp(\mu)(x)=\sum_{k=1}^\infty \frac{1}{k!} \mu^{\otimes k}\circ \tdelta^{(k-1)}(x)
=\frac{1}{n!}\mu^{\otimes n}\circ \tdelta^{(n-1)}(x)+0.\]
The second formula comes immediately.
\end{proof}

\begin{prop}
Let $\mu\in B_1^*$, such that $\lambda=\exp(\mu)$ is invertible for the product $\star$. Then 
\begin{align*}
\Phi&=\Psi_\mu\leftsquigarrow \lambda^{\star-1},&
\phi&=\mu \star \lambda^{\star -1}.
\end{align*}
\end{prop}

\begin{proof}
By Theorem \ref{theo3.12}, $\Psi_\mu=\Phi\leftsquigarrow \lambda$, so
\[\Psi_\mu\leftsquigarrow \lambda^{\star-1}=\Phi\leftsquigarrow (\lambda \star \lambda^{\star-1})=\Phi.\]
For any $x\in B$, $\phi(x)$, coefficient of $X$ in $\Phi(x)$, is given by
\[\phi(x)=\epsilon_\delta \circ \varpi_1\circ \Phi,\]
where $\varpi_1:\K[X]\longrightarrow \K X$ is the canonical projection. By homogeneity of $\Psi_\mu$,
\begin{align*}
\phi&=\epsilon_\delta \circ \varpi_1\circ (\Psi_\mu\otimes \lambda^{\star-1})\circ \delta\\
&=\epsilon_\delta \circ (\Psi_\mu\otimes \lambda^{\star-1})\circ (\varpi_1\otimes \id)\circ \delta\\
&=(\lambda \otimes \lambda^{\star -1})\circ (\pi_1\otimes \id)\circ \delta\\
&=((\lambda \circ \pi_1)\otimes \lambda^{\star-1})\circ \delta\\
&=(\mu\otimes \lambda^{\star-1})\circ \delta\\
&=\mu\star \lambda^{\star -1}. \qedhere
\end{align*}\end{proof}

\begin{example}
Let $\mu\in (\calH_\gr)_1$ defined by $\mu(\grun)=1$. Then $\exp(\mu)(\grun)=\mu(1)=1$.
In the same way as in \cite[Proposition 11]{Foissy37}, it is possible to prove that $\lambda=\exp(\mu)$ 
is invertible for the $\star$ product. Moreover, for any graph $G$ with $n$ vertices,
\[\lambda(G)=\frac{1}{n!}\sum_{\substack{V(G)=I_1\sqcup \ldots \sqcup I_n,\\ |I_1|=\ldots=|I_n|=1}}
\prod_{i=1}^n \mu(G_{\mid I_i}) =\frac{n!}{n!}=1.\]

Let $G$ be a graph, and let $\sim\in \eq_c(G)$. If $G$ is not connected, then $G/\sim$ has at least two vertices,
so $(\pi_1\otimes \id)\circ \delta(G)=0$, which implies (again) that $\phi_{chr}(G)=0$. 
Let us now assume that $G$ is connected. The unique $\sim_0\in \eq_c(G)$ such that $|G/\sim_0|=1$ is the equivalence
with only one class, which indeed belongs to $\eq_c(G)$ as $G$ is connected. Moreover, $G\mid \sim_0=G$. 
We obtain 
\[\phi_{chr}(G)=\mu\star \lambda^{\star-1}(G)=\mu(\grun)\lambda^{\star-1}(G)=\lambda^{\star-1}(G).\]
Therefore, for any connected graph $G$, $\lambda^{\star-1}(G)=\phi_{chr}(G)$, which entirely determines 
the character $\lambda^{\star-1}$. For any graph $G$,
\[\Phi_{chr}(G)=\sum_{\sim\in \eq_c(G)}
\Psi_\mu(G/\sim) \lambda^{\star-1}(G\mid \sim)
=\sum_{\sim\in \eq_c(G)}\left(\prod_{C\in V(G)/\sim}\phi_{chr}(G_{\mid C})\right) X^{\cl(\sim)},\]
where $\cl(\sim)$ is the number of classes of $\sim$.
\end{example}

\subsection{Morphisms to $\QSym$}

Let us recall the following result, due to Aguiar, Bergeron and Sottile \cite{Aguiar2006-2}:

\begin{prop}
Let $(B,m,\Delta)$ be a graded and connected bialgebra, and let $\lambda \in \chara(B)$. 
There exists a unique homogeneous bialgebra morphism $\Phi_\lambda:(B,m,\Delta)\longrightarrow
(\QSym,\squplus,\Delta)$ such that $\epsilon_\delta \circ \Phi_\lambda=\lambda$.
For any $x\in B_+$,
\[\Phi_\lambda(x)=\sum_{n\geq 1}\sum_{k_1,\ldots,k_n\in \N_{>0}}
\lambda^{\otimes k}\circ (\pi_{k_1}\otimes \ldots \otimes \pi_{k_n})\circ \tdelta^{(n-1)}(x)
(k_1\ldots k_n).\]
\end{prop}

\begin{prop}
Let $(B,m,\Delta,\delta)$ be a double bialgebra, such that $(B,m,\Delta)$ is a graded and connected bialgebra
in the sense of Definition \ref{defi5.1}.
\begin{enumerate}
\item If
\begin{align}
\label{eq3}
&\forall n\in \N,&\delta(B_n)\subseteq B_n\otimes B+\ker(\Phi_{\epsilon_\delta}\otimes \Phi_{\epsilon_\delta}),
\end{align}
then the unique homogeneous double bialgebra morphism from $(B,m,\Delta,\delta)$
to $(\QSym,\squplus,\Delta,\delta)$ is  $\Phi_{\epsilon_\delta}$. 
Otherwise, there is no homogeneous double bialgebra morphism from $B$ to $\QSym$. 
\item For any $\lambda\in \chara(B)$, the unique homogeneous bialgebra morphism $\Phi_\lambda:
(B,m,\Delta)\longrightarrow (\QSym,\squplus,\Delta)$ such that $\epsilon_\delta \circ \Phi_\lambda
=\lambda$ is $\Phi\leftsquigarrow \lambda$. 
\end{enumerate}
\end{prop}

\begin{proof}
1. \textit{Unicity}. Let $\Phi$ be such a morphism. Then $\epsilon_\delta\circ\Phi=\epsilon_\delta$.
By the unicity in Aguiar, Bergeron and Sottile's theorem, $\Phi=\Phi_{\epsilon_\delta}$. \\

1. \textit{Existence}. Let us first assume that (\ref{eq3}) holds, 
and let us prove that $\Phi=\Phi_{\epsilon_\delta}$ is a double bialgebra morphism.
Let us prove that for any $x\in B_n$, $\delta\circ \Phi(x)=(\Phi\otimes \Phi)\circ \delta(x)$
by induction on $n$. If $n=0$, we can assume that $x=1_B$, and then
\[\delta\circ \Phi(1_B)=1\otimes 1=(\Phi\otimes \Phi)\circ \delta(1_B).\]
Let us assume the result at all ranks $<n$. For any $x\in B_n$, as $\Phi:(B,m,\Delta)\longrightarrow
(\QSym,\squplus,\Delta)$ is a bialgebra morphism,
\begin{align*}
(\tdelta\otimes \id)\circ \delta \circ \Phi(x)&=\squplus_{1,3,24}\circ (\delta \otimes \delta)\circ \tdelta \circ \Phi(x)\\
&=\squplus_{1,3,24}\circ (\delta \otimes \delta)\circ (\Phi\otimes \Phi)\circ \tdelta(x)\\
&=\squplus_{1,3,24}\circ (\Phi\otimes \Phi\otimes \Phi\otimes \Phi)\circ (\delta \otimes \delta)\circ \tdelta(x)\\
&= (\Phi\otimes \Phi\otimes \Phi)\circ m_{1,3,24}\circ(\delta \otimes \delta)\circ \tdelta(x)\\
&= (\Phi\otimes \Phi\otimes \Phi)\circ (\tdelta \otimes \id)\circ \delta(x)\\
&=  (\tdelta \otimes \id)\circ(\Phi\otimes \Phi)\circ \delta(x).
\end{align*}
We use the induction hypothesis for the third equality, as 
\[\tdelta(x)\in \bigoplus_{i=1}^{n-1}B_i\otimes B_{n-i}.\]
Hence, $\delta \circ \Phi(x)-(\Phi\otimes \Phi)\circ \delta(x)\in \prim(\QSym)\otimes \QSym$.
Moreover, by homogeneity of $\Phi$,
\[\delta\circ \Phi(x)\in \delta(\QSym_n)\subseteq \QSym_n\otimes \QSym.\]
By (\ref{eq3}),
\[(\Phi\otimes \Phi)\circ \delta(x)\in (\Phi\otimes \Phi)(B_n\otimes B+\ker(\Phi\otimes \Phi))
=\Phi(B_n)\otimes \Phi(B)\subseteq \QSym_n\otimes \QSym,\]
so finally
\[\delta \circ \Phi(x)-(\Phi\otimes \Phi)\circ \delta(x)\in \prim(\QSym)\cap \QSym_n\otimes \QSym
=\K(n)\otimes \QSym,\]
and we put $\delta \circ \Phi(x)-(\Phi\otimes \Phi)\circ\delta(x)=(n)\otimes y$.
As $\epsilon_\delta((n))=1$,
\begin{align*}
y&=(\epsilon_\delta\otimes \id)\circ \delta \circ \Phi(x)
-(\epsilon_\delta\otimes \id)\circ (\Phi\otimes \Phi)\circ \delta(x)\\
&=\Phi(x)-(\epsilon_\delta \otimes \Phi)\circ \delta(x)\\
&=\Phi(x)-\Phi(x)=0,
\end{align*}
so finally $\delta \circ \Phi(x)=(\Phi\otimes \Phi)\circ \delta(x)$.\\

Let us now assume that $\Phi$ is a double bialgebra morphism. Let $n\in \N$. For any $x\in B_n$,
as $\Phi$ is homogeneous, 
\[(\Phi\otimes \Phi)\circ \delta(x)=\delta\circ \Phi(x)\in \delta(\QSym_n)\subseteq \QSym_n\otimes \QSym.\]
Let us put 
\[\delta(x)=\sum_{k,l\geq 0}x_{k,l},\]
where $x_{k,l}\in B_k\otimes B_l$ for any $k,l$. Then, by homogeneity of $\Phi$,
\begin{align*}
(\Phi\otimes \Phi)\circ \delta(x)&=\sum_{k,l\geq 0} \underbrace{(\Phi\otimes \Phi)(x_{k,l})}
_{\in \QSym_k\otimes \QSym_l}\in \QSym_n\otimes \QSym.
\end{align*}
Therefore, if $k\neq n$, $x_{k,l}\in \ker(\Phi\otimes \Phi)$ and we finally obtain that
$x\in B_n\otimes B+\ker(\Phi\otimes \Phi)$. \\

2. Let $\lambda \in \chara(B)$. Then $\Phi\leftsquigarrow \lambda$ is a bialgebra morphism. 
For any $n\in \N$,
\[\Phi\leftsquigarrow \lambda(B_n)=(\Phi \otimes \lambda)\circ \delta(B_n)
\subseteq (\Phi \otimes \lambda)(B_n \otimes B)\subseteq \Phi(B_n)\subseteq \QSym_n,\]
so $\Phi\leftsquigarrow \lambda$ is homogeneous. Moreover,
\[\epsilon_\delta\circ (\Phi\leftsquigarrow \lambda)
=(\epsilon_\delta \circ \Phi)\star \lambda=\epsilon_\delta\star \lambda=\lambda.\]
Hence, $\Phi\leftsquigarrow \lambda=\Phi_\lambda$. 
\end{proof}

\begin{remark}
Let $(B,m,\Delta,\delta)$ be a connected double bialgebra and let $\Phi$ be a double bialgebra morphism from
$(B,m,\Delta,\delta)$ to $(\QSym,m,\Delta,\delta)$. Then the unique double bialgebra morphism 
from $(B,m,\Delta,\delta)$ to $(\K[X],m,\Delta,\delta)$ is $\Phi_\QSym\circ \Phi$, 
where $\Phi_\QSym$ is described in Remark \ref{rem3.1}.
\end{remark}

\begin{remark}
Non homogeneous double bialgebra morphisms from $B$ to $\QSym$ may exist.
For example, the algebra morphism $\Psi:\K[X]\longrightarrow \QSym$ sending $X$ to $(1)$
is a double bialgebra morphism. By composition with the double bialgebra morphism from $B$ to $\K[X]$,
we obtain non homogeneous double bialgebra morphisms from $B$ to $\QSym$.
\end{remark}

\begin{remark}
The hypothesis (\ref{eq3}) does not hold if $B=\calH_\gr$. For example,
\begin{align*}
(\Phi_{\epsilon_\delta}\otimes \Phi_{\epsilon_\delta})\circ \delta(\grdeux)
&=(\Phi_{\epsilon_\delta}\otimes \Phi_{\epsilon_\delta})(\grun \otimes \grdeux+\grdeux \otimes\grun\grun)\\
&=(1)\otimes 2(11)+2(11)\otimes (2(11)+(2)),\\
\delta\circ \Phi_{\epsilon_\delta}(\grdeux)&=\delta(2(11))\\
&=(2)\otimes 2(11)+2(11)\otimes (2(11)+(2)).
\end{align*}
A way to correct this is to work with decorated graphs, see \cite{Foissy36} for more details. 
\end{remark}

\textbf{Acknowledgments}: the author acknowledges support from the grant ANR-20-CE40-0007
\emph{Combinatoire Alg\'ebrique, R\'esurgence, Probabilit\'es Libres et Op\'erades}. The author also warmly thanks Darij Grinberg for his careful reading and his helpful comments. 

\bibliographystyle{amsplain}
\addcontentsline{toc}{section}{References}
\bibliography{biblio}

\end{document}